\documentclass[11pt]{article}
\usepackage[utf8]{inputenc}
\usepackage{placeins}
\usepackage{booktabs} % For prettier tables
\usepackage{makecell} % Allows for multiline cells
\usepackage{array}

\usepackage{enumerate}
\usepackage{eufrak}
\usepackage{lscape}
\usepackage{rotating}
\usepackage{multirow}
\usepackage{color}
\usepackage{url}
\usepackage{xcolor}
\usepackage{amsthm}

\newtheorem{theorem}{Theorem}[section]

\usepackage[ruled,vlined,noline,linesnumbered]{algorithm2e}

\usepackage{titlesec}
\usepackage[titletoc]{appendix}

\newtheorem{corollary}{Corollary}[section]

\newtheorem{definition}{Definition}[section]

\newtheorem{lemma}{Lemma}[section]

\newtheorem{proposition}{Proposition}[section]
\newtheorem{remark}{Remark}[section]

\newtheorem{assumption}{Assumption}[section]

\newtheorem{subassumption}{Assumption\normalfont}[assumption]

\newenvironment{assumption*}
{\ifnum\value{subassumption}=0 \stepcounter{assumption}\fi\subassumption}
{\endsubassumption}
\newenvironment{assumption+}[1]
{\renewcommand{\thesubassumption}{#1}\subassumption}
{\endsubassumption}

\def\R{\mathbb{R}}
\def\Rplus{\mathbb{R}_{\ge 0}}
\def\T{\top}

\def\bbD{\mathbb{D}}
\def\cO{\mathcal{O}}
\newcommand{\n}[1]{\|#1\|}

\usepackage[breaklinks=true, pdftex, 
pdfborder={0 0 0}, colorlinks=true, linkcolor=blue, citecolor=blue, urlcolor=blue, 
bookmarks=false, pdfpagelabels=false]{hyperref}
\usepackage{fullpage}
\usepackage{times}
\usepackage{ifthen}   % for conditionals
\usepackage{graphicx,graphics}
\usepackage{setspace}
\usepackage{amssymb,amsmath,amsthm}
\usepackage{dsfont}
\usepackage{color}
\usepackage{wrapfig}
\usepackage{multirow}
\usepackage{multicol}
\usepackage{pdfpages}
\usepackage{longtable} 	
\usepackage{booktabs}
\usepackage{datetime}
\usepackage{xspace}
\usepackage{paralist} %% This gives enumeration environments & compactenum 
\usepackage{enumitem} % For tighter lists
\usepackage[top=1in,bottom=1in,left=0.95in,right=0.95in]{geometry} % Set margins
\usepackage[T1]{fontenc}
\usepackage{doi}
\usepackage[numbers,compress,sort]{natbib}
\usepackage{algorithmic}
\usepackage{comment}
\usepackage{cleveref} % for cref
\usepackage{fontawesome5}
\usepackage{listings}
\usepackage{upgreek}
\usepackage[bbgreekl]{mathbbol}
\DeclareSymbolFontAlphabet{\mathbb}{AMSb}
\DeclareSymbolFontAlphabet{\mathbbl}{bbold}
\usepackage{lineno}
\definecolor{codegreen}{rgb}{0,0.6,0}
\definecolor{codegray}{rgb}{0.5,0.5,0.5}
\definecolor{codepurple}{rgb}{0.58,0,0.82}
\definecolor{backcolour}{rgb}{0.95,0.95,0.92}
\definecolor{ForestGreen}{RGB}{34,139,34}
\usepackage{xcolor}
\usepackage{listings}
\usepackage{xparse}

\usepackage{stackengine}   %  By Joseph
\newcommand{\ubar}[1]{\stackunder[1.0pt]{$#1$}{\rule{.8ex}{.075ex}}}    %  By Joseph
\NewDocumentCommand{\codeword}{v}{%
	\texttt{\textcolor{blue}{#1}}%
}
\lstdefinestyle{mystyle}{
	backgroundcolor=\color{backcolour},   
	commentstyle=\color{codegreen},
	keywordstyle=\color{magenta},
	numberstyle=\tiny\color{codegray},
	stringstyle=\color{codepurple},
	basicstyle=\ttfamily\footnotesize,
	breakatwhitespace=false,         
	breaklines=true,                 
	captionpos=b,                    
	keepspaces=true,                 
	numbers=left,                    
	numbersep=5pt,                  
	showspaces=false,                
	showstringspaces=false,
	showtabs=false,                  
	tabsize=2
}
\usepackage{titlesec}
\usepackage{mdframed}  % creating color text boxes
\usepackage{colortbl}
\usepackage{mathrsfs} % for \mathscr
\definecolor{highlight}{HTML}{DFEFF9}
\mdfdefinestyle{highlightbox}{leftmargin=0pt,rightmargin=0pt,%
	innerleftmargin=3pt,innerrightmargin=2pt,%
	innertopmargin=4pt,innerbottommargin=4pt,%
	backgroundcolor=highlight,linecolor=white}
\definecolor{algm}{HTML}{D4EFDF} % background color
\mdfdefinestyle{highlightalg}{leftmargin=0pt,rightmargin=0pt,%
	innerleftmargin=3pt,innerrightmargin=2pt,%
	innertopmargin=4pt,innerbottommargin=4pt,%
	backgroundcolor=algm,linecolor=white}
\definecolor{soft}{HTML}{F2D7D5} % background color
\mdfdefinestyle{highlightstruct}{leftmargin=0pt,rightmargin=0pt,%
	innerleftmargin=3pt,innerrightmargin=2pt,%
	innertopmargin=4pt,innerbottommargin=4pt,%
	backgroundcolor=soft,linecolor=white}
\usepackage{soul}

\makeatletter
\def\namedlabel#1#2{\begingroup
	#2%
	\def\@currentlabel{#2}%
	\phantomsection\label{#1}\endgroup
}
\makeatother

\newcommand{\cm}{\mathrm{cm}}

\newcommand{\E}[2][]{\operatorname{\mathbb{E}}_{#1}\left[ #2\right]} % Exp value
 
\usepackage{bm} % For bold math symbols

\newcommand{\ab}{\bm{a}}

\newcommand{\db}{\bm{d}}

\newcommand{\gb}{\bm{g}}

\newcommand{\sba}{\bm{s}} 

\newcommand{\ub}{\bm{u}}

\newcommand{\vb}{\bm{v}}

\newcommand{\xb}{\bm{x}}

\newcommand{\yb}{\bm{y}}

\usepackage{tikz}
\usetikzlibrary{positioning}
\usetikzlibrary{calc}
\usetikzlibrary{shapes.geometric}
\usetikzlibrary{topaths}
\usetikzlibrary{external}
\usetikzlibrary{fadings,decorations.pathreplacing} 

\usepackage{pgfplots}
\usetikzlibrary{shapes,decorations.pathmorphing,patterns}
\pgfplotsset{compat=newest}
\usepgfplotslibrary{fillbetween}

\graphicspath{{figures/}}

\usepackage{hyperref}
\hypersetup{
unicode = false,
pdftoolbar = true,
pdfmenubar = true,
pdffitwindow = true,
pdftitle = {DS Survey},
pdfauthor = {Dzahini/Rinaldi/Royer/Zeffiro},
pdfsubject = {DirectSearch},
pdfnewwindow = true,
pdfkeywords = {DirectSearch,DFO},
colorlinks = true,
linkcolor = blue,
citecolor = blue,
filecolor = black,
urlcolor = blue,
breaklinks = true
}

\newcommand{\Z}{\mathbb{Z}}
\newcommand{\Esp}{\mathbb{E}}

\newcommand{\snOne}{\mathbb{S}^{n-1}}

\newcommand{\N}{\mathbb{N}}

\newcommand{\F}{\mathcal{F}}

\newcommand{\pr}{\mathbb{P}}
\newcommand{\M}{\mathcal{M}}

\newcommand{\vzero}{\mathbf{0}}

\newcommand{\abs}[1]{\left\lvert#1\right\rvert}

\newcommand{\accolade}[1]{\left\lbrace#1\right\rbrace}
\newcommand{\accoladekinN}[1]{{\left\lbrace#1\right\rbrace}_{k\in\mathbb{N}}}

\newcommand{\prob}[1]{\mathbb{P}\left(#1\right)}

\newcommand{\norme}[1]{\left\lVert#1\right\rVert}
\newcommand{\scal}[2]{\left\langle#1,#2\right\rangle}

\newcommand{\normii}[1]{{\left\lVert#1\right\rVert}_{2}}

\newcommand{\norminf}[1]{{\left\lVert#1\right\rVert}_{\infty}}

\newcommand{\BBD}{\stackunder[1.0pt]{$\bbD$}{\rule{1.0ex}{.075ex}}}

\newcommand{\PP}{{\stackunder[1.0pt]{$\bm{P}$}{\rule{1.0ex}{.075ex}}}}
\newcommand{\Sk}{\bm{{\stackunder[1.0pt]{$\bm{s}$}{\rule{.8ex}{.075ex}}}}_k}

\newcommand{\Xk}{\bm{{\stackunder[1.0pt]{$x$}{\rule{1.0ex}{.075ex}}}}_k}

\newcommand{\xk}{\bm{{x}}_k}

\newcommand{\xkun}{\bm{{x}}_{k+1}}

\newcommand{\x}{\bm{x}}
\newcommand{\y}{\bm{y}}
 
\newcommand{\rn}{\mathbb{R}^n}

\newcommand{\rnn}{\mathbb{R}^{n\times n}}

\newcommand{\rnp}{\mathbb{R}^{n\times p}}

\newcommand{\s}{\bm{s}}

\newcommand{\vi}{\bm{v}}
\newcommand{\SMatrix}{\bm{{\stackunder[1.0pt]{$S$}{\rule{1.3ex}{.075ex}}}}}

\newcommand{\di}{\bm{d}}

\newcommand{\f}{f}

\newcommand{\ef}{\varepsilon_f}

\newcommand{\ssize}{\alpha}% Can be changed to delta

\newcommand{\Dm}{{\stackunder[1.0pt]{$\ssize$}{\rule{.8ex}{.075ex}}}_k^m}
\newcommand{\dm}{\ssize_k^m}
\newcommand{\dpl}{\ssize_k^p}
\newcommand{\Mk}{\mathcal{M}_k}
\newcommand{\dmun}{{\ssize_{k+1}^m}}
\newcommand{\dpun}{{\ssize_{k+1}^p}}

\newcommand{\Dkp}{\mathbb{D}_k^p}

\newcommand{\Fok}{{\stackunder[0.6pt]{$f$}{\rule{.8ex}{.075ex}}}_k^{\bm{0}}}
\newcommand{\Fsk}{{\stackunder[0.6pt]{$f$}{\rule{.8ex}{.075ex}}}_k^{\s}}

\newcommand{\Dplj}{{\stackunder[1.0pt]{$\ssize$}{\rule{.8ex}{.075ex}}}_k^p}
\newcommand{\bl}{\color{black}}
\newcommand{\rd}{\color{black}}
\newcommand{\rdj}{\color{black}}
\newcommand{\bll}{\color{purple}}

\newcommand{\rev}{\textcolor{black}}
\newcommand{\revn}{\textcolor{black}}
\newcommand{\revnn}{\textcolor{black}}

\title{\revn{Direct-search methods in the year 2025:\\
Theoretical guarantees and algorithmic paradigms}}
\author{
K. J. Dzahini\thanks{Argonne National Laboratory, 9700 S. Cass Avenue, Lemont, IL 60439, USA \texttt{(kdzahini@anl.gov)}. Funding of this author was supported by the CAMPA and ComPASS-4, projects of the U.S. Department of Energy, Office of Science, Office of Advanced Scientific Computing Research and Office of High Energy Physics, Scientific Discovery through Advanced Computing (SciDAC) program under Contract No.~DE-AC02-06CH11357.}
\and 
F. Rinaldi\thanks{Dipartimento di Matematica ``Tullio Levi-Civita'', Universit\`a di Padova, Padua, Italy \texttt{(rinaldi@math.unipd.it)}. }
\and
C. W. Royer\thanks{LAMSADE, CNRS, Universit\'e Paris Dauphine-PSL, Place du Mar\'echal de Lattre de Tassigny, 75016 Paris,
France \texttt{(clement.royer@lamsade.dauphine.fr)}. Funding for this author's research was partially provided by Agence Nationale de la Recherche through program ANR-19-P3IA-0001
(PRAIRIE 3IA Institute) and by CNRS under the IAE grant BONUS.}
\and
D. Zeffiro\thanks{Dipartimento di Matematica ``Tullio Levi-Civita'', Universit\`a di Padova, Padua, Italy \texttt{(zeffiro@math.unipd.it)}.}
}
\begin{document}
	
\maketitle

\begin{abstract}
    Optimizing a function without using derivatives is a challenging paradigm, that 
    precludes from using classical algorithms from nonlinear optimization, and may 
    thus seem intractable other than by using heuristics. \rev{Nevertheless,} the 
    field of derivative-free optimization has succeeded in producing algorithms 
    that do not rely on derivatives and yet are endowed with convergence guarantees. 
    One class of such methods, called \revn{direct-search methods}, is particularly 
    popular thanks to its simplicity of implementation, even though its theoretical 
    underpinnings are not always easy to grasp.

    In this work, we survey contemporary direct-search algorithms from a 
    theoretical viewpoint, with the aim of highlighting the key theoretical 
    features of these methods. \rev{We provide a basic introduction to the main 
    classes of direct-search methods, including line-search techniques that have 
    received little attention in earlier surveys. We also put a particular emphasis 
    on probabilistic direct-search techniques and their application to noisy 
    problems, a topic that has undergone significant algorithmic development in 
    recent years. Finally, we complement existing surveys by reviewing the main 
    theoretical advances for \revn{solving} constrained and multiobjective 
    optimization using \revn{direct-search algorithms}.}
\end{abstract}

%%%%%%%%%%%%%%%%%%%%%%%%%%%%%%%%%%%%%%%%%%%%%%%%%%%%%%%%%%%%%%%%%%%%%%%%%%%%%%%%%%%%
\section{Introduction}
\label{sec:intro}
%%%%%%%%%%%%%%%%%%%%%%%%%%%%%%%%%%%%%%%%%%%%%%%%%%%%%%%%%%%%%%%%%%%%%%%%%%%%%%%%%%%%

Direct-search methods, whose name was coined by Hooke and Jeeves~\cite{HoJe61a}, 
form one of the main classes of derivative-free optimization (DFO) algorithms. 
Those methods work in a very easy and intuitive way: they  sample the objective 
function at a finite number of points at each iteration and decide which actions to 
take next solely based on those function values, either by function value 
comparison through ranking  or based on numerical values, and without \revn{fundamentally relying upon} any explicit 
or implicit derivative approximation or model building. This simple principle partly explains why direct-search approaches {\rdj have been} highly popular {\rdj since} the early days of numerical optimization. Despite the existence of tools that {\rdj allow to} easily 
use first- or second-order derivative information in an algorithm (i.e., tools that 
compute derivatives automatically), there is still \rev{a need for methods that 
rely solely on function values in simulation-based problems where derivative 
information can hardly be obtained~\cite{AuHa2017}. As a result, direct-search 
methods remain a popular choice of practitioners. Although classical techniques 
such as Nelder--Mead's simplex method~\cite{nelder1965simplex} or evolutionary 
algorithms \cite{back1996evolutionary} are not built on strong theoretical 
foundations, modern direct-search algorithms are typically equipped with 
convergence guarantees, that certify their ability to converge to a 
(possibly approximate) solution of the problem at hand \rev{asymptotically (global 
convergence) or in finite time (complexity bounds)}. Still, those guarantees are 
somewhat scattered in the literature, in research papers as well as surveys and 
books on derivative-free optimization.}

% Distinction with previous surveys/books
Indeed, direct-search methods have been the subject of several dedicated 
surveys~\cite{audet2014survey,KoLeTo03a,LeToTr00a,wright1996direct}.
The 1996 survey of Wright~\cite{wright1996direct} was arguably the first on 
direct-search schemes, with an emphasis on simplicial methods that may lack 
convergence guarantees. The later survey of Lewis et al.~\cite{LeToTr00a} 
discussed in particular pattern search methods along with the convergence results 
that relied on simple decrease. The landmark survey of Kolda et 
al.~\cite{KoLeTo03a} followed and provided a thorough investigation of convergence 
results on \revn{direct-search schemes} based on simple and sufficient decrease, 
highlighting developments in the (linearly) constrained setting but focusing mainly 
on smooth objectives. To the best of our knowledge, the most recent survey article 
is due to Audet~\cite{audet2014survey}, who described the developments of 
\revn{direct-search methods} with an emphasis on mesh adaptive \revn{direct-search} 
(MADS) \revn{variants and their} use in applications.
More broadly, direct-search methods have been featured in textbooks and reviews on 
derivative-free optimization. Conn et al.~\cite{CoScVibook} covered the basics of 
direct-search methods with convergence results. Audet and Hare~\cite{AuHa2017} 
focused on direct-search schemes based on simple decrease, and provided a thorough 
introduction to pattern search and MADS. Finally, in the recent overview of 
derivative-free algorithms, Larson et al.~\cite{LMW2019AN} provided \rev{complexity 
results for direct-search methods in comparison with zeroth-order algorithms, 
especially in the convex setting.}

\rev{This survey aims at complementing existing work by reviewing the main 
direct-search techniques proposed in the literature and their theoretical 
guarantees. First, more emphasis is put on line-search techniques, that 
have been somewhat overlooked in previous surveys. Secondly, connections are drawn 
between the directional and the mesh-based direct-search techniques, two classes 
of methods that tend to be discussed separately. We believe that readers should be 
aware of the two classes of algorithms, along with their specifics, that are 
highlighted throughout our work. Third, this survey begins with a thorough 
treatment of direct-search methods in the unconstrained setting, in both a smooth 
and a nonsmooth setting (note that we do not leverage convexity in this survey's 
results). The underlying goal is to ease the transition from this setting to more 
complicated setups, such as noisy or constrained optimization. We believe that a 
reader interested in using direct-search methods \revn{would} likely face such challenges, 
and that a good understanding of an algorithm's behavior in the simplest possible 
setting provides valuable insight regarding its extension to other settings.}

\rev{The rest of the survey is structured as follows. Section~\ref{sec:algosdef} 
provides a generic direct-search template for unconstrained problems, along with 
the key mathematical concepts used to establish convergence results. 
Section~\ref{sec:unc} reviews the main classes of algorithms in the unconstrained setting and represents a starting point for any of the subsequent sections. 
Section~\ref{sec:stoch} describes algorithms that tolerate noise in the function 
values. Section~\ref{sec:cons} is concerned with handling constrained formulations. 
Section~\ref{sec:multiobj} presents the main results in a multiobjective setting. 
The final comments are given in Section~\ref{sec:conc}.}

%%%%%%%%%%%%%%%%%%%%%%%%%%%%%%%%%%%%%%%%%%%%%%%%%%%%%%%%%%%%%%%%%%%%%%%%%%%%%%%%%%%%
\section{Main \rev{algorithm} and definitions}
\label{sec:algosdef}
%%%%%%%%%%%%%%%%%%%%%%%%%%%%%%%%%%%%%%%%%%%%%%%%%%%%%%%%%%%%%%%%%%%%%%%%%%%%%%%%%%%%

\paragraph{Notations}
Throughout, \rev{the} vectors {\rdj are} written in lowercase boldface (e.g.,
$\vi\in\rn, n\geq 2$) while matrices {\rdj are} written in uppercase boldface (e.g., 
$\bm{S}\in\rnp$). The set of column vectors of a matrix $\bm{D}$ {\rdj is} denoted 
by~$\mathbb{D}$. \rev{Other sets {\rdj are} denoted by calligraphic or upper Greek letters (e.g., $\mathcal{K},\Omega$), except for the sets~$\mathbb{N}$ and 
$\mathbb{R}$ of natural and real numbers}. The set of nonnegative real numbers {\rdj are} denoted by $\Rplus$. Given two integers $a,b$, we write $[a:b]$ 
for the set of all integer values $a \le z \le b$. Sequences indexed by $\N$ \revn{are} denoted by $\{a_k\}_{k \in \N}$ or $\{a_k\}$ in absence of ambiguity. Random 
\revn{objects} {\rdj are} denoted using underlined letters (e.g., 
$\ubar{z}, \ubar{\vi}, \SMatrix$). The expected value operator {\rdj is} denoted by 
$\mathbb{E}[\cdot]$. When necessary, we \revn{specify} the variables on which the 
expected value is computed (e.g. $\mathbb{E}_{\ubar{\zeta}}[\cdot]$ for the 
expected value over $\ubar{\zeta}$). \rev{ Given a point $\xb \in \R^n$ and a set 
$\Omega \subset \R^n$, we use \rev{$\pi_{\Omega}(\cdot)$} for the projection 
operator on $\Omega$. The notation $B_{\delta}(\xb)$ refers to the ball of radius 
$\delta$ centered at $\xb$. Finally, given} a quantity $A$, the notation 
$\mathcal{O}(A)$ means a constant times $A$, where the constant does not depend on 
$A$.

%%%%%%%%%%%%%%%%%%%%%%%%%%%%%%%%%%%%%%%%%%%%%%%%%%%%%%%%%%%%%%%%%%%%%%%%%%%%%%%%%%%%
\subsection{The direct-search framework}
\label{ssec:framework}

In this section, we introduce the main ideas behind \revn{direct-search schemes} 
by considering an unconstrained minimization problem
\begin{equation}
\label{eq:uncpb}
    \min_{\xb \in \R^n} f(\xb),
\end{equation}
where \rev{$f:\R^n \rightarrow \R$} is the objective function of our problem. 
\rev{In this survey, we describe theoretical guarantees that require some regularity 
on $f$ (continuity, smoothness). However,} we emphasize that a direct-search method 
only requires access to \rev{the} values of $f$ to proceed.

Algorithm~\ref{alg:basicDS} describes a generic direct-search framework, that \revn{is} used as a baseline for developing more involved variants in the subsequent 
sections. The method initially selects a point in the variable space (also called 
the \emph{incumbent solution} in the direct-search literature), as well as an 
initial value for the so-called \emph{stepsize}. {\rdj At every iteration, the function  is evaluated at trial points, obtained from the current point by a displacement along certain directions called \emph{poll directions}.}
%, of the current stepsize. 
If one of those points provides a decrease with respect to the current function 
value as measured by a \emph{forcing function}, this point can be accepted as the 
new point, and the stepsize can be increased (or kept constant). Otherwise, 
the next point is set to be the current point, and the stepsize is decreased.

\begin{algorithm}[h]
	\caption{Basic direct-search method}
	\label{alg:basicDS}
	\begin{algorithmic}[1]
        \par\vspace*{0.1cm}
		\STATE \textbf{Inputs:} Starting point/incumbent solution 
		$\x_0\in\mathbb{R}^{n}$, initial stepsize $\ssize_0> 0$,\\ 
        stepsize update parameters $0 < \theta < 1 \le \gamma$, forcing 
        function $\rho: \Rplus \rightarrow \Rplus$.
		\FOR{$k=0,1,2,\ldots$}
			\STATE Select a set $\bbD_k$ of poll directions.
		      \IF{$f(\xk+\bm{s}) < f(\xk) - \rho(\ssize_k)$ holds for 
		      some $\bm{s} \in \accolade{\ssize_k\bm{d}:\bm{d} \in \bbD_k}$}
                \STATE Declare the iteration as successful,
                set $\xkun=\xk+\bm{s}$ and $\ssize_{k+1}= \gamma \ssize_k$.
            \ELSE
			    \STATE Declare the iteration as unsuccessful, set $\xkun=\xk$ and 
                $\ssize_{k+1}=\theta \ssize_k$.
            \ENDIF
		\ENDFOR
	\end{algorithmic}
\end{algorithm}

Each step of Algorithm~\ref{alg:basicDS} highlights the key ingredients that define 
a given direct-search method. First, the choice of poll directions $\mathbb{D}_k$ 
is instrumental to deriving convergence results. These directions can be chosen 
in a deterministic or randomized fashion, fixed throughout the iterations or 
varying at every iteration. In the rest of Section~\ref{sec:algosdef}, we \revn{review} the most classical polling strategies.

Secondly, the condition for success
\begin{equation}
\label{eq:basicDSdec}
         f(\xk+\bm{s}) < f(\xk) - \rho(\alpha_k),
\end{equation}
characterizes a major divide between direct-search methods, depending on whether or 
not any possible improvement in the objective function is accepted. 
\rev{Algorithm~\ref{alg:basicDS} (or an instance thereof) is said to be based on a 
\emph{sufficient decrease condition} when \revnn{$\rho(\alpha)>0$ if $\alpha>0$}. Otherwise, 
the algorithm is said to be based on \emph{simple decrease}.}

Thirdly, in order to check the decrease condition~\eqref{eq:basicDSdec}, 
Algorithm~\ref{alg:basicDS} requires access to the function value at both 
$\xk+\bm{s}$ and $\xk$. When exact values of $f$ are available, one only needs to 
evaluate $f$ once at a given point, hence values can be saved between iterations to 
avoid extra calculations. On the contrary, repeated evaluations (also termed 
replications in the literature) are particularly useful in the context of 
\emph{stochastic} optimization, where only noisy evaluations of the objective are 
available. Such algorithms are an important part of this survey and are reviewed in 
Section~\ref{sec:stoch}.

Fourthly, the set of points that are considered at every iteration may be broader 
than those obtained through the polling directions and the current stepsize. In 
particular, \revn{direct-search methods of line-search type}, which are discussed in 
Section~\ref{ssec:linesearch}, evaluate additional points beyond those considered 
in polling in order to perform extrapolation. 

Finally, the decision made at every iteration of Algorithm~\ref{alg:basicDS} may 
not require all points in $\accolade{\alpha_k\bm{d}:\bm{d} \in \bbD_k}$ to be 
checked in case of success (although this remains necessary to declare an iteration 
as unsuccessful).  \rev{Algorithm~\ref{alg:basicDS} (or an instance thereof) is 
said to use \emph{opportunistic polling} if {\rdj S}tep~2 stops as soon as it finds a 
vector~$\bm{s}$ satisfying~\eqref{eq:basicDSdec}. Otherwise, the algorithm is said 
to use \emph{complete polling}.}

\begin{remark}\label{rem:searchstep}
	\rev{Practical direct-search methods typically include an optional search step, 
	that allows for using problem-specific heuristics as well as additional strategies 
	for optimization~\cite[Section 7.2]{AuHa2017}. Nevertheless, theoretical 
	guarantees of direct-search techniques do not rely on the properties of this search 
	step. Since our survey focuses on these guarantees, we present direct-search 
	methods without a search step, but emphasize that search steps are often crucial 
	to the practical performance of \revn{these algorithms}.}
\end{remark}

%%%%%%%%%%%%%%%%%%%%%%%%%%%%%%%%%%%%%%%%%%%%%%%%%%%%%%%%%%%%%%%%%%%%%%%%%%%%%%%%%%%%
\subsection{Main mathematical concepts}
\label{ssec:maths}

In this section, we gather key algebraic definitions associated with the choice 
of directions in a direct-search method, for future reference in the rest of the 
survey. We also provide standard analytical definitions associated with the 
objective function (and the constraint functions when applicable).

%%%%%%%%%%%%%%%%%%%%%%%%%%%%%%%%%%%%%%%%%%%%%%%%%%%%%%%%%%%%%%%%%%%%%%%%%%%%%%%%%%%%%
\subsubsection{Geometry of polling sets}
\label{sssec:geometrypoll}

One of the main features of a direct-search method is its strategy to choose the 
poll directions. Those directions play a key role in guaranteeing 
convergence, provided they possess favorable geometric properties. The first 
property below is an intrinsic property of a direction set.

\begin{definition}
\label{def:cm}{\rdj (\cite[Section~2.1]{GrRoViZh2015})}
	For a vector $\vi \in \R^n\setminus \{0\}$ and a set of directions $\bbD$, 
    the cosine measure of $\bbD$ with respect to~$\vi$ is defined as
	\begin{equation*}
		\textnormal{cm}(\bbD, \vi) 
		:= 
		\max_{\db \in \bbD} \frac{\scal{\db}{\vi}}{\n{\db}\n{\vi}} \, . 
	\end{equation*}
	The \emph{cosine measure} of $\bbD$ is then 
	\begin{equation*}
		\textnormal{cm}(\bbD) 
		:= 
		\min_{\vi \in \mathbb{R}^n \setminus \{0\}} \textnormal{cm}(\bbD, \vi).
	\end{equation*}
\end{definition}

The cosine measure of a set $\bbD$ captures how much of the space is covered by the 
cone spanned using the vectors in $\bbD$. The best scenario corresponds to spanning 
the entire space, and corresponds to the following notion.

\begin{definition}
\label{def:pss}
	A set $\bbD$ of \revn{$n_D$} vectors in $\R^n$ is called a \emph{positive spanning set} if its 
	cosine measure is positive, or equivalently, if every vector in $\R^n$ can be 
	written as a nonnegative linear combination of vectors in $\bbD$.
\end{definition}

Positive spanning sets are the typical choice of poll directions for direct-search 
methods. In addition to relying on positive spanning sets, certain direct-search 
methods ensure that the points considered for evaluation lie on carefully designed 
grids. We adopt here the terminology of meshes and frames~\cite{AuHa2017}, and 
provide the definitions independently of any algorithm.

\begin{definition}
\label{def:meshframe}
%Consider a positive spanning set $\bbD$, and let $\bm{D} \in \R^{n \times n_D}$ be a matrix representation of $\bbD$. 
{\rdj Let $\bm{D}:=\bm{G}\bm{Y}\in\R^{n\times n_D}$ be the matrix representation of a set $\bbD$, where $\bm{G}\in\rnn$ is invertible and the columns of $\bm{Y}\in\Z^{n\times n_D}$ form a positive spanning set for $\rn$.
}
For any vector $\bm{x} \in \R^n$ and any 
    $\alpha^{\rdj m}>0$, the \emph{mesh} of coarseness $\alpha^{\rdj m}$ corresponding to $\bm{D}$ 
    and centered at $\bm{x}$, is defined by
    \begin{equation}
    \label{eq:mesh}
        \mathcal{M} = 
        \accolade{\x+\alpha^{\rdj m} \bm{D}\bm{z}:\  \bm{z}\in\N^{n_D}} {\rdj \subset \rn}.
    \end{equation}
% For any $\delta > 0$, 
{\rdj T}he frame {\rdj $\mathcal{F}\subseteq \mathcal{M}$} of {\rdj extent} {\rdj $\alpha^p\geq \alpha^m$ generated by $\bm{D}$, centered at $\x$,} is defined by
    \begin{equation}
    \label{eq:frame}
        \mathcal{F} = 
        \accolade{\bm{y} \in \mathcal{M}: \norminf{\bm{y}-\x} \le {\rdj \alpha^p d_{\max}}}{\rdj ,}
    \end{equation}
{\rdj where $d_{\max}:=\max\accolade{\norminf{\di'}:\di'\in \bbD}$.}
\end{definition}

One interesting feature of meshes and frames is that they lead to asymptotic 
density, a concept that has proven fundamental in analyzing direct-search methods, 
especially in the nonsmooth setting~\cite{AuDe2006,AuHa2017}.

\begin{definition}%\cite[Definition~2.2]{AuDe2006} and \cite[Definition~8.5]{AuHa2017}
\label{defi:asymptdense} 
    \rev{An infinite sequence $\{\bbD_k\}_{k \in \mathcal{K}}$} of vector sets 
    defines an asymptotically dense set of directions if the set
    \[
        {\rdj \bbD_\infty:=} \accolade{\di/\norminf{\di}:\ \di\in\bbD_k }_{k\in\mathcal{K}}
    \]
    is dense in the unit sphere {\rdj $\mathbb{S}^{n-1}$ of} $\rn${\rdj , \revn{i.e. $\forall \bm{u}\in \mathbb{S}^{n-1}, \forall \epsilon>0, \exists \di_\infty\in \bbD_\infty$ such that $\normii{\bm{u}-\di_\infty}<\epsilon$}}.
\end{definition}

%%%%%%%%%%%%%%%%%%%%%%%%%%%%%%%%%%%%%%%%%%%%%%%%%%%%%%%%%%%%%%%%%%%%%%%%%%%%%%%%%%%%
\subsubsection{Function classes and stationary points}
\label{sssec:func}

By design, a direct-search method can be applied to any function for which the 
objective value (or approximations thereof, see Section~\ref{sec:stoch}) can be 
queried. Still, in order to establish convergence \rev{and convergence rates}, one 
typically assumes a certain level of regularity in the objective function. For this 
reason, we \revn{assume} continuity of the objective function, though we note that 
\revn{direct-search methods have} \rev{also} been applied to extended-value and 
discontinuous problems~{\rdj \cite{AudBatKoj2022,bouchet2023theorie,custodio2012analysis}}. 
\revn{In this survey, we \revn{consider} both Lipschitz continuous functions or functions 
with Lipschitz continuous derivative, in the sense of the two 
definitions below.}

\begin{definition}
\label{def:FLip}
    A function \revn{$F: \R^n \rightarrow \R^{n'}$ with $n' \in \{1,n\}$}is said to be \emph{$L$-Lipschitz continuous} {\rdj for $L>0$} if
    \[
        \forall (\xb,\yb) \in \rn\times\rn, \quad 
        \|F(\xb)-F(\yb)\| \le L \|\xb-\yb\|.
    \]
\end{definition}

\begin{definition}
\label{def:FLsmooth}
    A function $f:\R^n \rightarrow \R$ is said to be \emph{$L$-smooth} for $L>0$ if 
    $f$ is continuously differentiable and $\nabla f:\R^n \rightarrow \R^{n}$ is 
    \revn{\emph{$L$-Lipschitz continuous.}}
\end{definition}

When optimizing a smooth function without constraints, direct-search methods 
generally aim at converging towards a point with zero gradient, called a 
\emph{stationary point}. We \revn{present} results of that form in the upcoming 
sections. Existing analyzes of direct-search methods also provide guarantees on 
nonsmooth functions using generalized derivatives. The next two definitions 
\rev{from Clarke differential calculus~\cite{Clar83a} correspond to the most 
classical notions of generalized derivatives} used in the direct-search literature.

\begin{definition}
\label{def:clarkeder}
	\rev{The \emph{Clarke directional derivative} of $f$ at $\xb \in \rn$ in the 
	direction $\vi \in \rn$} is
	\begin{equation}
		f^{\circ}(\xb;\, \vi) 
		= 
		\limsup_{\substack{\yb \rightarrow \xb \\ t \searrow 0}} 
		\frac{f(\yb + t\vi) - f(\yb)}{t} \, .
	\end{equation} 
    The \revn{\emph{Clarke subdifferential}} of $f$ at $\xb$ is then defined as
    \begin{equation}
    \label{eq:clarkesubdiff}
        \partial f(\xb) = \left\{ 
        \gb \in \R^n
        \ \middle|\ 
        \gb^\T \vi \le f^{\circ}(\xb;\, \vi)\quad \forall \vi \in \R^n \right\}.
    \end{equation}
\end{definition}

\begin{definition}
\label{def:clarkestatio}
	The point $\xb^*$ is \emph{Clarke stationary} for $f$ if 
	$f^\circ(\revn{\xb^*}{\bl;}\, \vi) \geq 0$ for every $\vi \in \mathbb{R}^n$ 
    or, equivalently, if $\vzero \in \partial f(\revn{ \xb^*})$.
\end{definition}

The notion of Clarke stationarity is weaker than many others used in variational 
analysis. \rev{As a result, it has been used extensively to derive guarantees for 
direct-search schemes~\cite{audet2002analysis}. More recently, an alternate notion 
of stationarity due to Goldstein~\cite{goldstein1977optimization} has been employed 
to obtain complexity results for direct-search methods.}

\begin{definition}
\label{def:goldsteinstatio}
    Given $\delta>0$, the point $\xb^*$ is {called} \emph{$\delta$-Goldstein 
    stationary} for $f$ if 
    \[
        \inf\left\{ \|\gb\|
        \ \middle|\ 
        \gb \in \mathrm{conv}\left(
        \cup_{\yb \in B_{\delta}(\revn{\xb^*})} \partial f(\yb)\right) \right\} = 0
    \]
    where $\mathrm{conv}\left(\mathcal{S}\right)$ denotes the set of convex 
    combinations of vectors in $\mathcal{S}$.
\end{definition}

Note that this survey \revn{also considers} approximate definitions of stationary 
points in order to derive complexity results. For instance, in the smooth setting,
rather than a point with zero gradient, we might be interested in converging 
towards a point with small enough gradient norm. In the nonsmooth setting, the 
notion of an approximate stationary point is more intricate, and we provide the 
definition below.

\begin{definition}
\label{def:approxgoldstein}
    Given $\delta>0$ and $\varepsilon>0$, a vector $\xb \in \R^n$ is called a 
    {\rdj \emph{$(\delta, \varepsilon)$-Goldstein stationary}} point for a function $f$ if
    \begin{equation}
    \label{ed:approxgoldstein}
	   \inf\left\{\n{\gb} 
	   \ \middle|\ 
	   \gb \in \text{conv} \cup_{\yb \in B_{\delta}(\xb)} \partial f(\yb) 
	   \right\} 
	   \le \varepsilon.
    \end{equation}
\end{definition}

The concept of approximate stationary points \revn{are also} 
discussed in Section~\ref{ssec:directional}.

%%%%%%%%%%%%%%%%%%%%%%%%%%%%%%%%%%%%%%%%%%%%%%%%%%%%%%%%%%%%%%%%%%%%%%
\subsection{Basic convergence ingredients}
\label{ssec:basiccv}

In this section, we provide the main assumptions and technical lemmas 
that are used while analyzing a direct-search method. As in the rest 
of Section~\ref{sec:algosdef}, we focus on the unconstrained, 
deterministic setting, and provide results for both smooth and 
nonsmooth objectives. 

%%%%%%%%%%%%%%%%%%%%%%%%%%%%%%%%%%%%%%%%%%%%%%%%%%%%%%%%%%%%%%%%%%%%%%
\subsubsection{Main assumptions}
\label{sssec:basicass}

We begin by listing the key assumptions that are made on the objective 
function. 

\begin{assumption} \label{ass:fbound}
    The objective $f$ is lower bounded, i.e., there exists 
    $f^* \in \R$ such that {\rdj $f(\xb) \geq  f^*$} $\forall \xb \in \R^n$.
\end{assumption}  

Assumption~\ref{ass:fbound} is typically used together with a guarantee 
of decrease at every iteration of the algorithm to establish convergence.

\begin{assumption} \label{ass:L0}
	The sublevel set 
	$L_0 := \{\xb \in \mathbb{R}^n \ | \ f(\xb) \leq f(\xb_0) \}$ is 
	bounded, where $\xb_0 \in \R^n$ is the \revn{initial} point of 
	Algorithm~\ref{alg:basicDS}.
\end{assumption}

Assumption~\ref{ass:L0} certifies that the iterates of 
Algorithm~\ref{alg:basicDS} remain in a bounded set, which is particularly 
helpful to guarantee convergence of a subsequence of the iterates.

In addition to assumptions on the objective function, an analysis of 
Algorithm~\ref{alg:basicDS} typically requires conditions on the polling 
directions, that are stated below.
\begin{assumption} \label{ass:dirlowup}
	For $\db \in \mathbb{D}_k$, $k \in \mathbb{N}_0${\rdj ,} we have 
	$\n{\db} \in [d_{\min}, d_{\max}]$ where $d_{\min}$ and $d_{\max}$ are positive 
	constants.
\end{assumption}
The upper bound ensures that the information obtained in unsuccessful steps is 
relevant for the directional derivative, while the lower bound ensures that the 
sufficient decrease condition can indeed be satisfied for good descent directions. 

The last classical assumption concerns the forcing function.
\begin{assumption} \label{ass:rho}
	The forcing function $\rho: \Rplus \rightarrow \Rplus$, is continuous, 
	\rev{nondecreasing and} $\rho(\ssize) = o(\ssize)$ as $\ssize \rightarrow 0$.
\end{assumption}
Note that using simple decrease (i.e., $\rho \equiv 0$) leads to immediate 
satisfaction of Assumption~\ref{ass:rho}. Popular choices for $\rho$ in direct 
search based on sufficient decrease include $\rho(\alpha)=c\,\alpha^q$ with $c>0$ 
and $q >1$.

%%%%%%%%%%%%%%%%%%%%%%%%%%%%%%%%%%%%%%%%%%%%%%%%%%%%%%%%%%%%%%%%%%%%%%%%%%%%%%%%%%%%
\subsubsection{Key lemmas}
\label{sssec:basiclemmas}

A remarkable feature of direct-search schemes is that their analysis 
relies upon \emph{unsuccessful iterations}, at which the iterate does 
not change. On those iterations, one can relate a stationarity 
criterion of interest to the stepsize parameter. In a smooth setting, 
the following lemma is the first step towards establishing theoretical 
guarantees of a direct-search technique.

\begin{lemma} \label{l:re}
    Let $f$ be an $L$-smooth function, and consider that 
    Algorithm~\ref{alg:basicDS} is applied to the minimization of~$f$. 
    Suppose that the $k$th iteration of Algorithm~\ref{alg:basicDS} is 
    unsuccessful. Then, \rev{for any $\db \in \bbD_k$,}
	\begin{equation} \label{eq:deltak}
		- \scal{\nabla f(\xb_k)}{\db} 
		\le
		\ssize_k \frac{L\n{\db}^2}{2} + \frac{\rho(\ssize_k)}{\ssize_k} 
		\,.
	\end{equation}
\end{lemma}
\begin{proof}
	\rev{
	Since iteration $k$ is unsuccessful, for any $\db \in \bbD_k$, the decrease 
	condition is not satisfied, and thus 
	\begin{equation}
	\label{eq:nodec}
		f(\xb_k+\alpha_k \db)-f(\xb_k) \ge -\rho(\alpha_k).
	\end{equation}
	Since $f$ is $L$-smooth by assumption, {\rdj using for example~\cite[Lemma~9.4]{AuHa2017} yields} 
	\[
		f(\xb_k+\alpha_k \db) 
		\le 
		f(\xb_k) + \alpha_k\scal{\nabla f(\xb_k)}{\db} 
		+ \frac{L\alpha_k^2}{2}\n{\db}^2.
	\]
	Combining this inequality with~\eqref{eq:nodec} and dividing by $\alpha_k>0$ 
	yields the desired result.
	}
\end{proof}

Lemma~\ref{l:re} highlights the importance of choosing a good set of polling 
directions, so as to relate the left-hand side of~\eqref{eq:deltak} with 
a convergence criterion. The use of positive spanning sets is particularly 
suited in that regard, as it leads to the following result.

\begin{proposition} \label{prop:unsuccPSSsmooth}
    Let the assumptions of Lemma~\ref{l:re} hold{\rdj , in particular that the $k$th iteration of Algorithm~\ref{alg:basicDS} is 
    unsuccessful}. 
    Suppose further that Assumption~\ref{ass:dirlowup} holds, and that 
    $\bbD_k$ is a positive spanning set with cosine measure $\tau>0$. Then,
	\begin{equation}
    \label{eq:unsuccPSSsmooth}
		\n{\nabla f(\xb_k)} 
		\le 
		\frac{1}{\tau}\left[ \ssize_k \frac{L d_{\max}}{2} 
		+ \frac{\rho(\ssize_k)}{\ssize_k\,d_{\min}} \right] \, .
	\end{equation}
\end{proposition}

The result of Proposition~\ref{prop:unsuccPSSsmooth} can be used to 
guarantee convergence of the gradient norm to zero for a subsequence of 
unsuccessful iterates. We \revn{illustrate} such a result in 
Section~\ref{ssec:directional}.

\rev{
In absence of smoothness, the result of Lemma~\ref{l:re} no longer holds. However, 
one can still guarantee that 
\begin{equation} \label{eq:unsuccgen}
	\frac{f(\xb_k+\alpha_k \db)-f(\xb_k)}{\alpha_k}  \ge \revn{-}\frac{\rho(\ssize_k)}{\ssize_k}.
\end{equation}
for any $\db \in \bbD_k$.}
The left-hand side of~\eqref{eq:unsuccgen} resembles the formula of the Clarke 
derivative, and is used together with the notion of refining subsequence to yield 
asymptotic guarantees on the Clarke directional derivatives. We define the former 
concept below.

\begin{definition}
\label{def:refining}
    Consider a sequence of iterates $\{\xb_k\}_{k \in \N}$, a sequence of stepsize 
    parameters $\{\ssize_k\}_{k \in \N}$ and a sequence of polling sets 
    $\{\bbD_k\}_{k \in \N}$ associated with a direct-search algorithm. Let 
    $\mathcal{K} \subset \N$ be an {\rdj infinite} index subsequence such that the $k$th iteration 
    is unsuccessful for every $k \in \mathcal{K}$ and the subsequence 
    $\{\xb_k\}_{k \in \mathcal{K}}$ converges towards $\revn{ \hat{\xb}} \in \R^n$. 
    \begin{itemize}
        \item The subsequence $\{\xb_k\}_{k \in \mathcal{K}}$ is called a 
        \emph{refining subsequence} if  $\lim_{k \in \mathcal{K}} \alpha_k = 0$, in 
        which case ${\bll \hat{\xb}}$ is called a refined point.
        \item If $\{\xb_k\}_{k \in \mathcal{K}}$ is a refining subsequence, a 
        direction $\db \in \R^n$ is called a \emph{refining direction} if 
        there exists an infinite subset ${\rdj \mathcal{L}} \subset \mathcal{K}$ and a sequence 
        $\{\db_{\rev{\ell}}\}_{\rev{\ell} \in \mathcal{L}}$ such that 
        $\db_{\rev{\ell}} \in \bbD_{\rev{\ell}}$ for any 
        $\rev{\ell} \in \mathcal{L}$ and\\ 
        $\lim_{\rev{\ell} \in \revn{\mathcal{L}}} \tfrac{\db_{\rev{\ell}}}{\|\db_{\rev{\ell}}\|} 
        = \tfrac{\db}{\|\db\|}$.
    \end{itemize}
\end{definition}

By reasoning on refining subsequences, one can directly establish asymptotic results on (Clarke) directional derivatives. This property is crucial in analyzing algorithms based on simple decrease such as those described in 
Section~\ref{ssec:meshbased}.

\begin{lemma} \label{l:cvrefineddir}
    Suppose that $f$ is Lipschitz continuous, and consider that 
    Algorithm~\ref{alg:basicDS} is applied to the minimization of~$f$ with 
    a forcing function that satisfies Assumption~\ref{ass:rho}. 
    Let $\mathcal{K} \subset \N$ be an \revn{infinite} index subsequence such that 
    $\{\xb_k\}_{k \in \mathcal{K}}$ is a refining subsequence with refined point 
    $\revn{\hat{\xb}}$. Then, for any refining direction $\db$, it holds that
	\begin{equation}
    \label{eq:cvrefineddir}
		f^\circ(\revn{\hat{\xb}}; \db) \geq 0 \, .
	\end{equation}
\end{lemma}
\begin{proof}
	\rev{
	Let $\{\db_{\ell}\}_{\ell \in \mathcal{L}}$ be a sequence of directions such that 
	$\db_{\ell} \in \bbD_{\ell}$ for every $\ell \in \mathcal{L}$ 
	and $\lim_{\ell \in \mathcal{L}} \frac{\db_{\ell}}{\|\db_{\ell}\|} = 
	\frac{\db}{\|\db\|}$. 
	Without loss of generality, suppose that all iterations in $\mathcal{L}$ denote 
	unsuccessful iterations. Applying~\eqref{eq:unsuccgen} at iteration 
	$\ell \in \mathcal{L}$ to the direction $\db_{\ell}$ gives
	\[
		\frac{f(\xb_{\ell}+\alpha_{\ell}\db_{\ell})-f(\xb_{\ell})}{\alpha_{\ell}} 
		\ge 
		-\frac{\rho(\alpha_{\ell})}{\alpha_{\ell}}.
	\]
	Taking the limit on both sides as $\ell$ goes to infinity, we have 
	$\db_{\ell} \rightarrow \db^*$, 
	$\alpha_{\ell} \rightarrow 0$ as well as $\xb_{\ell} \rightarrow {\revn{\hat{\xb}}}$. 
	It follows that
	\[
		f^\circ(\revn{ \hat{\xb} ;}\, \db) \ge \lim_{\ell \rightarrow \infty} 
		\frac{f(\xb_{\ell}+\alpha_{\ell}\db_{\ell})-f(\xb_{\ell})}{\alpha_{\ell}} 
		\ge 0.
	\]
	}
\end{proof}

This lemma {\rdj is} used in the subsequent analyzes of 
\revn{direct-search algorithms} in the nonsmooth setting.

%%%%%%%%%%%%%%%%%%%%%%%%%%%%%%%%%%%%%%%%%%%%%%%%%%%%%%%%%%%%%%%%%%%%%%%%%%%%%%%%%%%%
\section{Algorithms with guarantees for unconstrained optimization}
\label{sec:unc}
%%%%%%%%%%%%%%%%%%%%%%%%%%%%%%%%%%%%%%%%%%%%%%%%%%%%%%%%%%%%%%%%%%%%%%%%%%%%%%%%%%%%

\rev{In this section, we review the main instances of Algorithm~\ref{alg:basicDS} 
that possess theoretical guarantees. We begin by reviewing directional algorithms, 
arguably the simplest class of such instances. We then present a simple mesh-based 
algorithm, and present its convergence guarantees in the unconstrained setting, 
that follow from standard analyzes of MADS-type algorithms. Finally, we review the 
theory behind line-search algorithms.}

%%%%%%%%%%%%%%%%%%%%%%%%%%%%%%%%%%%%%%%%%%%%%%%%%%%%%%%%%%%%%%%%%%%%%%%%%%%%%%%%%%%%
\subsection{Directional algorithms}
\label{ssec:directional}

\rev{Directional direct-search methods directly follow the template of 
Algorithm~\ref{alg:basicDS}. Those schemes differ in their choice of polling 
directions, polling strategy, and forcing function. Under standard 
assumptions, a directional method generally comes with global convergence 
guarantees~\cite{KoLeTo03a}. Variants with sufficient decrease are also amenable 
to a complexity analysis~\cite{Vicente2013W}, and we stress out that sufficient 
decrease is currently the only paradigm in which complexity guarantees have been 
obtained.
} 

\rev{We first review the main results for deterministic directional 
\revn{direct-search methods} on smooth problems in 
Section~\ref{sssec:directionaldetermsmooth}. We then describe the corresponding 
results for probabilistic directional \revn{direct-search schemes} in 
Section~\ref{sssec:directionalrandomsmooth}. 
Finally, we cover existing guarantees for deterministic methods on nonsmooth 
problems in Section~\ref{sssec:directionaldetermnonsmooth}.}

%%%%%%%%%%%%%%%%%%%%%%%%%%%%%%%%%%%%%%%%%%%%%%%%%%%%%%%%%%%%%%%%%%%%%%%%%%%%%%%%%%%%
\subsubsection{Deterministic convergence results in the smooth case}
\label{sssec:directionaldetermsmooth}

We begin the analysis of Algorithm~\ref{alg:basicDS} by focusing on the smooth 
setting. In that context, the polling sets are typically required to be positive 
spanning sets with positive cosine measure. We formalize this requirement through 
the following assumption.

\begin{assumption}
\label{ass:cmDk}
    The polling sets $\{\bbD_k\}$ used in Algorithm~\ref{alg:basicDS} satisfy 
    $\cm(\bbD_k) \ge \tau$ for every $k \in \N$, where $\tau>0$.
\end{assumption}

A direct application of Proposition~\ref{prop:unsuccPSSsmooth} under this 
assumption yields the following convergence result.

\begin{theorem}  \label{th:liminfdirds}
    Suppose that Algorithm~\ref{alg:basicDS} is applied to an $L$-smooth function 
    $f$ under Assumptions~\ref{ass:fbound} and~\ref{ass:rho}. \rev{Suppose also 
    that the sequence $\{\bbD_k\}$ satisfies Assumptions~\ref{ass:dirlowup} 
    and~\ref{ass:cmDk}.}
    \rev{Let $\mathcal{K}$ be a index set of unsuccessful iterations such that 
    $\lim_{k\in \mathcal{K}} \alpha_k \rightarrow 0$. Then,}
	\begin{equation}
    \label{eq:liminfdirds}
		\lim_{k \in \mathcal{K}}\n{\nabla f(\xb_k)} = 0.
	\end{equation}
\end{theorem}
{\rdj
\begin{proof}
All the assumptions made in Proposition~\ref{prop:unsuccPSSsmooth} hold whence taking the limit when $k\to\infty, k\in \mathcal{K}$ in~\eqref{eq:unsuccPSSsmooth} yields~\eqref{eq:liminfdirds} since $\rho(\alpha_k)/\alpha_k\to 0$ per Assumption~\ref{ass:rho}.
\end{proof}
}

Theorem~\ref{th:liminfdirds} is a weak convergence result, in that it only 
guarantees that a subsequence of gradient norms is going to zero provided there 
exists a subsequence of stepsizes that converges to $0$. Note that neither this 
result nor that of Lemma~\ref{l:re} do  leverage the use of sufficient decrease. 
A first step towards improving the results of Theorem~\ref{th:liminfdirds} 
consists {\rdj of} exploiting sufficient decrease to guarantee that the entire sequence 
of stepsizes converges to $0$.

\begin{lemma}
\label{l:cvalphasuffdec}
    Suppose that Algorithm~\ref{alg:basicDS} is applied to an $L$-smooth function 
    $f$ under {\rdj Assumptions~\ref{ass:fbound}, \ref{ass:dirlowup} and~\ref{ass:rho}}. 
    Suppose further that the algorithm uses sufficient decrease. Then, 
    $\lim_{k \rightarrow \infty} \alpha_k = 0$.
\end{lemma}

Combining Lemma~\ref{l:cvalphasuffdec} with Theorem~\ref{th:liminfdirds} does not 
guarantee that $\lim_{k \rightarrow \infty} \n{\nabla f(\xb_k)}=0$, since the 
result of Theorem~\ref{th:liminfdirds} only concerns subsequences of unsuccessful 
iterations. \rev{A popular approach to obtain these results requires complete 
polling to be performed~\cite{KoLeTo03a,CoScVibook}, which corresponds to our 
assumptions in Theorem~\ref{th:lipgglobal} below. Note that similar results can 
also be obtained in the context of locally strongly convex 
objectives~\cite{KoLeTo03a,DoLeTo03a}, or in the context of line-search 
methods~\cite{lucidi2002global}, to be discussed in Section \ref{ssec:linesearch}.}

\begin{theorem} \label{th:lipgglobal}
    Let the assumptions of Lemma~\ref{l:cvalphasuffdec} hold. In addition, 
    suppose that Algorithm~\ref{alg:basicDS} is applied using \rev{complete 
    polling} and a sequence of polling sets satisfying Assumption~\ref{ass:cmDk}. 
    Then,
    \begin{equation*}
        \lim_{k \rightarrow \infty} \n{\nabla f(\xb_k)} = 0.
    \end{equation*} 
\end{theorem}

\paragraph{Complexity guarantees} In addition to global convergence results, the 
sufficient decrease condition also endows direct-search schemes with complexity 
bounds, or non-asymptotic results, that quantify how many iterations or function 
evaluations are needed in the worst case to reach an approximate optimality 
criterion, such as $\|\nabla f(\xb_k)\| \le \epsilon$ (for nonconvex objectives) or 
$f(\xb_k)-\min_{\xb \in \R^n} f(\xb) \le \epsilon$ (for convex and strongly convex 
objectives). \revn{The first bounds of this nature for direct-search schemes} were 
established by Vicente~\cite{Vicente2013W} in the nonconvex setting, and later 
extended to the convex and strongly convex setting by Dodangeh and 
Vicente~\cite{dodangeh2016worst}. A unified and simplified argument was presented 
in \cite{konevcny2014simple}, assuming however no expansion in successful steps and 
a suitable initialization of the algorithm. 

% A common feature of 
{\rdj One thing} these results {\rdj have in common} is that they exhibit additional dependencies on 
the problem dimension compared to those for derivative-based algorithms, a common 
trait among derivative-free algorithms~\cite{CCartis_NIMGould_PhLToint_2022}. 
\rev{In this survey, we focus on analyzing the nonconvex setting, and give the 
corresponding result below.}

\begin{theorem}
%{\rdj (Help with reference please? It looks like~\cite{dodangeh2016optimal} cite below provides a particular case...)} 
\label{t:n2complexity}
    Let the assumptions of Theorem~\ref{th:lipgglobal} hold. 
    Suppose further that $\rho(\alpha) = c\,\alpha^2$ for some $c > 0$, 
    and that $|\bbD_k| \leq m$ for 
    every $k \in \mathbb{N}$. Then, for any 
    $\epsilon>0$, \rev{Algorithm~\ref{alg:basicDS} reaches an iterate 
    $\xb_k$ such that $\|\nabla f(\xb_k)\| \le \epsilon$ in at most 
    \begin{equation}
    \label{eq:n2complexity}
    	\cO\left(\tau^{-2}\epsilon^{-2}\right)
    	\ \mbox{iterations} 
    	\quad \mbox{and} \quad
    	\cO\left(m\tau^{-2}\epsilon^{-2}\right)
    	\ \mbox{function evaluations},
    \end{equation}
    where $\tau$ is the bound on the cosine measures given by 
    Assumption~\ref{ass:cmDk}.
    }
\end{theorem}

As stated, the result of Theorem~\ref{t:n2complexity} does not exhibit explicit 
dependencies on the problem dimension. However, if one aims at minimizing the 
complexity bounds as functions of $m$ and $\tau$, it has been shown that the best 
choice consists {\rdj of} selecting $\bbD_k$ as an orthonormal basis and its negative, in 
which case the results hold with $m=2\,n$ and 
$\tau=\tfrac{1}{\sqrt{n}}$~ (see, e.g., \cite{dodangeh2016optimal,vicente2013worst}). As a result, the iteration 
complexity bounds exhibit a dependency in $n$, while the evaluation complexity 
bounds have a dependency in $n^2$. In the next section, we \revn{describe} 
variants of directional \revn{direct-search algorithms} that \rev{attempt} to 
reduce that dependency through randomization.

\begin{remark}
\label{rem:dirdscvx}
	\rev{
	Under convexity assumptions on $f$, it is possible to improve the 
	dependency on $\epsilon$ from $\epsilon^{-2}$ to~$\epsilon^{-1}$ in 
	the convex case and $\log(\epsilon^{-1})$ in the strongly convex 
	case, thereby following similar improvements in the complexity of 
	gradient descent techniques~\cite{dodangeh2016worst}. Note however 
	that the dependencies on $m$ and $\tau$, which are intrinsic to the 
	use of a direct-search scheme, are identical to that in the 
	nonconvex setting.
	}
\end{remark}

%%%%%%%%%%%%%%%%%%%%%%%%%%%%%%%%%%%%%%%%%%%%%%%%%%%%%%%%%%%%%%%%%%%%%%%%%%%%%%%
\subsubsection{\rev{Probabilistic convergence results in the smooth case}}
\label{sssec:directionalrandomsmooth}

An important subclass of directional \revn{direct-search algorithms}, that has 
grown in importance over the past decade, \rev{relies on assuming \emph{probabilistic 
properties} of the polling sets rather than relying on deterministic properties 
such as Assumption~\ref{ass:cmDk}. The associated schemes} employ a 
small number of \rev{randomly chosen} directions to effectively reduce the 
number of points considered at every iteration \rev{compared to a deterministic 
approach from $\mathcal{O}(1)$ to $\mathcal{O}(n)$. We note that such randomized 
strategies have been applied to zeroth-order algorithms that attempt to build a 
gradient approximation, yielding improved theoretical guarantees compared to 
deterministic variants~\cite{nesterov2017random, bergou2020stochastic}.}

In the context of directional \revn{direct-search methods}, 
Gratton et al{\rdj .}~\cite{GrRoViZh2015} 
\rev{established theoretical guarantees for randomized approaches thanks to the 
notion of \emph{probabilistic descent}, that considers the polling sets as 
random \revn{objects}.}

\begin{assumption}\label{ass:DKprob}
    The random polling set sequence $\{{\rdj \BBD}_k \}$ 
    \rev{satisfies} 
    \begin{equation}
    \label{eq:DKprob}
        \rev{
        \pr( \cm({\rdj \BBD}_k, -\nabla f(\Xk)) 
        \geq \tau \ | \ \F_{k - 1} ) 
        \geq 
        \beta 
        \qquad \mbox{for all}\ k,
        }
    \end{equation}
   \rev{where $\tau>0$, $\beta \in (0,1)$, $\F_{k-1}$ is the $\sigma$-algebra 
   generated by $\BBD_0,\dots,\BBD_{k-1}$ and $\F_{-1}=\emptyset$.}
\end{assumption}
Note that Assumption~\ref{ass:DKprob} is clearly satisfied when 
Assumption~\ref{ass:cmDk} holds, i.e., when the polling sets are positive 
spanning sets. {\rdj We note that since the objective function is deterministic and $\BBD_0,\dots,\BBD_{k-1}$ are $\F_{k-1}$-measurable by construction, then $\nabla f(\Xk)$ is also $\F_{k-1}$-measurable and hence deterministic with respect to $\F_{k-1}$ since its distribution before the construction of $\BBD_k$ is only determined by $\BBD_0,\dots,\BBD_{k-1}$. When the $\BBD_i, i=0,1,\dots,k$, are mutually independent, then in particular $\BBD_k$ is independent of $\F_{k-1}$ so that for all~$k$, requiring $\prob{\cm(\BBD_k,\bm{v})\geq \tau}\geq \beta$ to hold for any arbitrarily fixed $\bm{v}\in\rn$ is enough to satisfy~\eqref{eq:DKprob}. When the elements in these independent sets $\BBD_i$ (which do not need to be positive spanning sets) are independent uniformly distributed random vectors on the unit sphere of~$\rn$, the latter inequality is shown in~\cite[Lemma~B.2]{GrRoViZh2015} to be satisfied, thereby providing a demonstration of the satisfiability of~\eqref{eq:DKprob} and its realisticness. We finally emphasize that a more formal and rigorous demonstration of how the above inequality leads to~\eqref{eq:DKprob} is provided by~\cite[Proposition~B.3]{GrRoViZh2015}.
}  
%\rev{More generally, randomly chosen directions can satisfy~\eqref{eq:DKprob} even when the polling sets are not positive spanning sets, since the property~\eqref{eq:DKprob} involves the cosine measure with respect to a single vector.}

Under Assumption~\ref{ass:DKprob}, one can derive a high-probability function 
evaluation complexity \rev{for Algorithm~\ref{alg:basicDS}, that implies 
almost-sure global convergence of the algorithmic process. Note that expected 
value results can also be obtained by that process~\cite{GrRoViZh2015}.}

\begin{theorem} \label{t:ncomplexity}
    Suppose that Algorithm~\ref{alg:basicDS} is applied to an $L$-smooth 
    function $f$ with $\rho(\alpha)=c\,\alpha^2$ for some $c>0$ and 
    $\gamma = \theta^{-1}$. Suppose that Assumption~\ref{ass:dirlowup} holds 
    for all realizations of $\{\BBD_k\}$, and that Assumption~\ref{ass:DKprob} 
    holds with $\beta > 1/2$ and $|\BBD_k| \leq m$ for every $k \in \mathbb{N}$. 
    Then, for any $\epsilon>0$, \rev{Algorithm~\ref{alg:basicDS} reaches an 
    iterate $\Xk$ such that $\|\nabla f(\Xk)\| \le \epsilon$ in at most 
    \begin{equation}
    \label{eq:ncomplexity}
    	\cO\left(\tau^{-2}\epsilon^{-2}\right) 
    	\quad \mbox{iterations} \quad 
        \mbox{and} \quad 
        \cO\left(m\tau^{-2}\epsilon^{-2}\right)
        \quad \mbox{function evaluations},
    \end{equation} 
    with probability at least $1-e^{\cO(\tau^{-2}\epsilon^{-2})}$.}
\end{theorem} 
When every $\BBD_k$ consists of a direction uniformly selected in the unit 
sphere \rev{and} its negative, probabilistic descent 
(Assumption~\ref{ass:DKprob}) is satisfied with $m = 2$ and 
\rev{$\tau = \cO(1/\sqrt{n})$}. 
\rev{It follows that the algorithm reaches an approximate stationary point in 
at most $\cO(n\,\epsilon^{-2})$ iterations \emph{and} function evaluations with 
probability $1-e^{\cO(n\,\epsilon^{-2})}$. The complexity bounds, that are 
probabilistic, improve over the deterministic result of 
Theorem~\ref{t:n2complexity} by a factor of $n$.}

\rev{Similar} results were also established for random search schemes in Bergou 
et al.~\cite{bergou2020stochastic}, in which a stepsize dependent on the 
Lipschitz constant of the gradient \rev{was} used. Direct-search methods 
employing a random direction and its negative at every iteration can be viewed 
as adaptive versions of those schemes, that maintain an estimate of the 
Lipschitz constant of the gradient through the stepsize parameter. The analysis 
of \revn{direct-search methods} using probabilistic descent directions was recently extended by 
Roberts and Royer~\cite{roberts2023direct} to allow for unbounded sets of 
descent directions, highlighting the connection with the theory of random 
matrices. In that work, each descent direction is defined through combinations 
of columns of a random matrix $\PP_k \in \mathbb{R}^{n \times r}$, that 
represents a subspace of possible directions at iteration $k$. {\rdj The complexity result from Gratton et al.~\cite{GrRoViZh2015} is recovered, with an additional dependence on the probabilistic bounds for the extreme singular values of~$\PP_k$.}
% The complexity result from Gratton et al.~\cite{GrRoViZh2015} is retrieved up to an additional dependence on probabilistic bounds for the extreme singular values of $\PP_k$. 
Such probabilistic bounds are non trivial to compute and have been 
the subject of recent studies~\cite{DzaWildHashing2022}.

%%%%%%%%%%%%%%%%%%%%%%%%%%%%%%%%%%%%%%%%%%%%%%%%%%%%%%%%%%%%%%%%%%%%%%%%%%%%%%%%%%%%
\subsubsection{Deterministic convergence results in the nonsmooth case}
\label{sssec:directionaldetermnonsmooth}

We now turn to the nonsmooth setting, where we merely assume that $f$ is 
Lipschitz continuous. Convergence results for \rev{deterministic} methods based on 
sufficient decrease require asymptotic density of the set of directions.

\begin{theorem} \label{t:cdir}
    Suppose that Algorithm~\ref{alg:basicDS} is applied to a Lipschitz 
    continuous function $f$. Let $\mathcal{K}$ be a sequence of unsuccessful steps such 
    that $\{\xb_k\}_{k \in \mathcal{K}}$ is a refining subsequence converging to $\revn{ \hat{\xb}}$. 
    If the corresponding set of refining directions is dense in the unit sphere, 
    then $\revn{\hat{\xb}}$ is a Clarke stationary point.
\end{theorem}

\begin{proof}
    By Lemma~\ref{l:cvrefineddir}, for any refining direction $\db$ associated 
    with $\{\xb_k\}_{k \in \mathcal{K}}$, we have $f^\circ(\revn{\hat{\xb}};\db) \ge 0$. Since the 
    set of refining directions is dense in the unit sphere, for any nonzero 
    vector $\ub \in \R^n$, there exists a sequence of refining directions 
    $\{\db_i\}_{i \in \N}$ such that $\db_i \rightarrow \tfrac{\ub}{\|\ub\|}$ 
    as $i \rightarrow \infty$. As a result, we obtain in the limit that 
    $f^\circ\left(\revn{ \hat{\xb}};\tfrac{\ub}{\|\ub\|}\right) \ge 0$. Since the Clarke 
    directional derivative is positively homogeneous\revn{~\cite[Proposition 2.1.1]{Clar83a}}, we also have 
    $f^\circ(\revn{\hat{\xb}};\ub) \ge 0$, from which the desired conclusion follows.
\end{proof}

Under additional assumptions, one can obtain a stronger result akin to 
Theorem~\ref{th:lipgglobal} for nonsmooth objectives. While this result follows 
from well established techniques, we are not aware of any existing reference in 
the literature. 

\begin{theorem}\label{t:as_density}
    Let the assumptions of Theorem~\ref{t:cdir} hold. Suppose further that 
    Assumptions \ref{ass:fbound}, \ref{ass:rho} and \ref{ass:dirlowup} hold. 
    For any $k \in \N$, define 
    $\hat{\bbD}_k = {\accolade{{\db}/{\n{\db}} \ | \ \db \in \bbD_k}}$. 
    If $\cup_{k \in \mathcal{K}} \hat{\bbD}_k$ is dense in the unit sphere for every 
    infinite subsequence $\mathcal{K}$, then every limit point of 
    $\{\xb_k\}_{k \in \N}$ is Clarke stationary.
\end{theorem}

Unlike in the smooth setting, complexity bounds for \revn{direct-search algorithms} 
remain rather unexplored in the nonsmooth setting. Garmanjani and 
Vicente~\cite{garmanjani2013smoothing} provided complexity results for a 
special class of nonsmooth \rev{objectives that can be approximated by a family 
of smooth functions $\{\tilde{f}(\cdot,\mu)\}_{\mu>0}$ such that 
$\tilde{f}(\cdot,\mu) \rightarrow f$ as $\mu \rightarrow 0$.} Examples of such 
functions include Gaussian or uniform smoothing, for which one can show that 
$\tilde{f}(\cdot,\mu)$ is 
$\mathcal{O}\left(\tfrac{\sqrt{n}}{\mu}\right)$-smooth~\cite{nesterov2017random}. 
We present below a simplified version of the analysis in Garmanjani and 
Vicente~\cite{garmanjani2013smoothing} based on a fixed smoothing parameter.

\begin{theorem}
\label{th:smoothingdirds}
    Let $f$ be a Lipschitz continous function satisfying 
    Assumption~\ref{ass:fbound}. Given $\epsilon>0$, suppose that we apply 
    Algorithm~\ref{alg:basicDS} to the smoothed function 
    $\tilde{f}(\cdot,\epsilon)$ \rev{under Assumption~\ref{ass:L0} with
    $\rho(\alpha)=c\,\alpha^{3/2}$ and polling sets satisfying 
    Assumptions~\ref{ass:dirlowup}, \ref{ass:cmDk} with 
    $\tau=\frac{1}{\sqrt{n}}$.}

    Then, the algorithm takes at most $\cO(n^{3/2} \epsilon^{-3})$ iterations 
    and $\cO(n^{5/2} \epsilon^{-3})$ function evaluations \rev{of the smoothed 
    objective} to obtain a point such that 
    $\n{\nabla_{\xb} \tilde{f}(\xb,\epsilon)} < \epsilon$.
\end{theorem}

Very recently, Lin et al{\rdj .}~\cite{lin2022gradient} obtained a function evaluation 
complexity of $\cO(n^{3/2}\epsilon^{-5})$ \rev{for} zeroth-order, 
smoothing-based methods to achieve a similar outcome \revn{to} that of 
Theorem~\ref{th:smoothingdirds}. The former result has a lower dependence 
\rev{on $n$ but a higher dependence from $\epsilon$ compared to that of the 
latter}. Note however that the algorithm of Lin et al.~\cite{lin2022gradient} 
only uses true function values rather than values for smoothed approximations.

%%%%%%%%%%%%%%%%%%%%%%%%%%%%%%%%%%%%%%%%%%%%%%%%%%%%%%%%%%%%%%%%%%%%%%%%%%%%%%%
\subsection{Mesh-based algorithms}
\label{ssec:meshbased}

In contrast with {\rdj several} directional direct-search techniques {\rdj using sufficient decrease}, mesh-based algorithms 
rely on simple decrease, and enforce convergence thanks to a careful 
generation of search directions. Pioneer{\rdj ing} methods of that form 
include generalized pattern search (GPS)~\cite{Torc97a,BoDeFrSeToTr99a} and 
frame-based methods~\cite{CoPr01a}. We focus our presentation on the 
\emph{mesh adaptive \revn{direct-search}} (MADS) \rev{paradigm}~\cite{AuDe2006}, that 
is nowadays dominant in the landscape of mesh-based algorithms, thanks in part 
to the practical success of the NOMAD software~{\rdj \cite{NomadV4}}. 

Algorithm~\ref{alg:mads}{\rdj , inspired by~\cite[Algorithm~8.1]{AuHa2017}}, describes a \rev{basic mesh-based approach for 
unconstrained optimization. A key feature of these methods is the use of two 
stepsize parameters called \emph{mesh size parameter} {\rdj denoted by~$\alpha_k^m$} and 
\emph{poll size parameter} {\rdj denoted by~$\alpha_k^p$}, {\rdj in order to} produce a rich set of trial points to be 
considered by the algorithm.}
\rev{The algorithm defines directions based on a single set $\bbD$ formed by 
the columns of $\bm{D}:=\bm{G}\bm{Y}\in\R^{n\times n_D}$, where $\bm{G}\in\rnn$ 
is invertible and the columns of $\bm{Y}\in\Z^{n\times n_D}$ form a positive 
spanning set. At iteration $k$, the mesh $\Mk$, and the frame $\F_k$ centered at 
$\xk\in\rn$  are defined respectively as}
\begin{equation}
\label{eq:meshk} 
    \Mk:=\underset{\x\in\mathcal{V}_k}{\bigcup}\accolade{\x+\dm\bm{D}\bm{z}:
    \  \bm{z}\in\N^{n_D}}\subset\rn 
    \quad\mbox{and}\quad
    \F_k:=\accolade{\x\in\Mk:\norminf{\x-\xk}\leq \dpl d_{\max}},
\end{equation}
where $\mathcal{V}_k$ is the set containing all the points where the objective 
function has been evaluated by the start of iteration~$k$, $\dm\leq \dpl$ for 
all~$k$, and $d_{\max}=\max\accolade{\norminf{\di'}:\di'\in \mathbb{D}}$, so 
that $\F_k\subseteq\Mk$. \revn{We note that the relation~\eqref{eq:meshk} justifies that the poll size parameter is sometimes referred to as {\it frame size parameter} in the literature.} Note that the mesh is conceptual in the sense that it is 
actually never explicitly constructed nor stored. \rev{Trial points are generated 
so as to belong to the mesh, by defining $\M_k$ according to $\mathcal{V}_k$ and 
using the directions in $\bbD$.}

\begin{algorithm}[h]
	\caption{\revn{Mesh-based direct-search algorithm}}
	\label{alg:mads}
	\begin{algorithmic}[1]
        \par\vspace*{0.1cm}
		\STATE \textbf{Inputs:} $\xb_0\in\rn$, $\alpha_0^p>0$ {\rdj and}  
		% $\alpha_0^m =\min\{\alpha_0^p,(\alpha_0^p)^2\}$, 
		$\theta \in (0,1) \cap \mathbb{Q}$.\\
		\FOR{$k=0,1,2,\ldots$}
			\STATE Select a positive spanning set  $\Dkp$ \revn{(e.g., by calling Algorithm~\ref{alg:polldir}) so as to 
            satisfy}\\ 
			$\xk+\dm \di \in \F_k \subseteq\Mk$ for all $\di\in \Dkp$, where 
			$\dm = \min\{\dpl,(\dpl)^2\}$.
            \IF {$f(\xk+\bm{s})< f(\xk)$ holds for some 
            $\bm{s} \in  \accolade{\dm \di:\di\in \Dkp}$} 
            	\STATE Declare the iteration as successful, 
                set $\xkun = \xk +\bm{s}$ {\rdj and}
                $\alpha_{k+1}^p = \theta^{-1}\alpha_k^p${\rdj .}
                % and 
                % $\alpha_{k+1}^m = \min\left\{\alpha_{k+1}^p,
                % \left(\alpha_{k+1}^p\right)^2\right\}$.
            \ELSE 
                \STATE Declare the iteration as unsuccessful, set $\xkun = \xk$ 
                 and $\alpha_{k+1}^p = \theta \alpha_k^p$.
            \ENDIF
		\ENDFOR
	\end{algorithmic}
\end{algorithm}

The goal of a \rev{mesh-based} iteration is to locate an {\it improved mesh point} by 
evaluating the objective at a finite number of points lying on the frame $\F_k$. 
\rev{Algorithm~\ref{alg:polldir} describes how such directions can be chosen in order 
to lie on the frame.}
In practice, we point out that additional points on $\Mk \setminus \F_k$ are often 
considered {\rdj during a} search step. When iteration $k$ of 
Algorithm~\ref{alg:mads} is unsuccessful, the frame $\F_k$ is called a {\it minimal 
frame} and $\xk$ is called a {\it minimal frame center}~\cite[Section~2]{AuDe2006}. 
In that case, we have $\dpun = \theta \dpl < \dpl$ and $\dmun < \dm$, but 
allowing the poll size parameter $\dm$ to decrease more rapidly than $\dpl$ leads 
to strong guarantees on the trial points that are considered by the 
\rev{method}~\cite{AuHa2017}.

\begin{algorithm}[h]
	\caption{Creating the set $\Dkp$ of poll directions}
	\label{alg:polldir}
	\begin{algorithmic}[1]
        \par\vspace*{0.1cm}
		\STATE \textbf{Inputs:} $\vi_k\in\rn$ with $\norme{\vi_k}=1$ and 
		$\dpl\geq\dm >0$.\\
		\STATE Use $\vi_k$ to create the Householder matrix 
		$\mathbf{H}_k=\bm{I}_n-2\vi_k{\vi_k}^\top
		=[\bm{h}_{1,k}\ \bm{h}_{2,k}\ \dots\ \bm{h}_{n,k}]\in\rnn$. \\
    	\STATE Define $\mathbb{B}_k=\{\bm{b}_{1,k},\bm{b}_{2,k},\dots,
    	\bm{b}_{n,k}\}$ with 
    	$\bm{b}_{j,k}=\text{round}\left(
    	\frac{\dpl}{\dm}\frac{\bm{h}_{j,k}}{\norminf{\bm{h}_{j,k}}}\right)
    	\in \mathbb{Z}^n$, $j=1,2,\ldots,n$. \\
    	\STATE Set the poll set as 
    	$\Dkp=\mathbb{B}_k\cup (-\mathbb{B}_k)$.
	\end{algorithmic}
\end{algorithm}

\rev{A first step towards establishing convergence of Algorithm~\ref{alg:mads} 
consists {\rdj of} observing that the coarseness level of the 
mesh vanishes asymptotically.}

\begin{theorem}(\cite[Theorem~8.1]{AuHa2017})\label{zerothOrdMADStheor}
Suppose that Algorithm~\ref{alg:mads} is applied to a function $f$ satisfying 
Assumption~\ref{ass:L0}. Then, the stepsize parameter sequences satisfy
\begin{equation}\label{dmdlZeroEq}
\underset{k\to\infty}{\liminf}\ \dm=\underset{k\to\infty}{\liminf}\ \dpl=0.
\end{equation}
\end{theorem}

The proof of Theorem~\ref{zerothOrdMADStheor} relies on the properties of the 
meshes rather than sufficient decrease, as well as the use of a rational 
parameter to update both \rev{$\dm$ and $\dpl$}~\cite[Lemmas 7.1-7.3]{AuHa2017}. 
In particular, the structure of the matrix $\bm{D}$ used in 
Algorithm~\ref{alg:mads} guarantees that the distance between distinct mesh 
points is lower bounded by a function of the mesh size parameter $\dm$.

The second \rev{argument in a mesh-based convergence theory relates to} the 
Clarke generalized directional derivative of the objective function $f$ at 
refined points \rev{in refining directions, in the sense of 
Definition~\ref{def:refining}}. Crucially, the notion of \revn{the} refining sequence 
uses $\{\dm\}$ as the stepsize parameters but $\{\Dkp\}$ as polling sets.
\rev{By combining Theorem~\ref{zerothOrdMADStheor} with Assumption~\ref{ass:L0}, 
one can then establish that there exists at 
\revn{least} one refining subsequence with  
a refined point~\cite[Theorem~7.6]{AuHa2017}. Meanwhile, the existence of a 
refining direction follows from the compactness of the closed unit ball along 
with Bolzano--Weierstrass's theorem~\cite[Section~8.3]{AuHa2017}. Overall, the 
following convergence result can be derived.}

\begin{theorem}
\label{theo:madscdir}{\rdj (\cite[Theorem~8.3]{AuHa2017})}
    Suppose that Algorithm~\ref{alg:mads} is applied to a Lipschitz continuous 
    function~$f$ satisfying Assumption~\ref{ass:L0}. Let $\{\xb_k\}_{k \in \mathcal{K}}$ 
    be a refining subsequence for the algorithm with refined point 
    $\hat{\xb} \in \R^n$, and let $\di \in \R^n$ be a refining direction for 
    $\hat{\xb}$. Then $f^{\circ}(\hat{\x};\di)\geq 0$.
\end{theorem}
Theorem~\ref{theo:madscdir} is a direct consequence of 
Lemma~\ref{l:cvrefineddir}, but \rev{only appeared in the more general, 
constrained setting throughout the mesh-based direct-search 
literature~\cite[Theorem~3.12]{AuDe2006}, \cite[Theorem~8.3]{AuHa2017}). As 
such, the result of Theorem~\ref{theo:madscdir} is relatively weak, in the 
sense that we have no information regarding the refining directions.} A key to 
strengthening \rev{Theorem~\ref{theo:madscdir}} lies in the fact that the mesh 
size parameter $\dm$ goes to zero faster than the frame size parameter~$\dpl$. 
\rev{By carefully designing the poll directions, it then becomes possible to 
guarantee asymptotic density of those directions.} 

\begin{theorem}~(\cite[Theorem~8.6]{AuHa2017})
\label{theo:madsdensedir}
The sequence of sets $\Dkp$ produced by Algorithm~\ref{alg:polldir} forms an 
asymptotically dense set of directions if the directions $\vi_k$ used to 
construct the Householder matrices $\mathbf{H}_k$ form a dense sequence 
$\accoladekinN{\vi_k}$ in the unit sphere $\snOne$.
\end{theorem}

\begin{corollary}
\label{cor:madsdense}
    Let the assumptions of Theorem~\ref{theo:madscdir} hold, and suppose that 
    the set of refining directions for $\hat{\x}$ is dense in the unit sphere. 
    Then $\hat{\x}$ is a Clarke stationary point.
\end{corollary}

Remarkably, a simpler version of the result leads to convergence to a 
stationary point in the presence of smoothness\rev{, which we again emphasize 
is not the usual setup of mesh-based analyzes}. 
%The proof of this result leverages sublinearity of the Clarke directional derivative.

\begin{corollary}
\label{cor:madssmooth}
    Under the assumptions of Theorem~\ref{theo:madscdir}, suppose further 
    that $f$ is smooth near a refined point $\hat{\x} \in \R^n$ \rev{and 
    that} the set of refining directions contains a positive spanning set.
    \rev{Then,} $\hat{\x}$ is a stationary point for $f$.
\end{corollary}

By design, a mesh-based algorithm aims at generating an asymptotically dense 
set of polling directions, and thus fulfills convergence conditions in both a 
nonsmooth and a smooth setting. In that sense, mesh-based algorithms enjoy 
asymptotic convergence guarantees for a \rev{broad} class of functions. 
Non-asymptotic results, such as complexity results or convergence rates, have 
not been derived for mesh-based methods that rely on simple decrease. Note that 
combining a mesh-based approach with sufficient decrease does lead to complexity 
guarantees, however this combination has only been adopted in the context of 
stochastic optimization (see Section~\ref{ssec:sds}), and remains nonstandard 
in the deterministic case.

%%%%%%%%%%%%%%%%%%%%%%%%%%%%%%%%%%%%%%%%%%%%%%%%%%%%%%%%%%%%%%%%%%%%%%%%%%%%%%%
\subsection{Line-search algorithms}
\label{ssec:linesearch}

Line-search-based derivative-free methods combine a sufficient decrease 
condition  with a thorough search-direction exploration procedure.  Those 
line-search approaches date back 40 years\rev{, with a first example 
of derivative-free line-search method proposed by De Leone, Gaudioso and 
Grippo~\cite{grippo1988global}, and later analyzed by Grippo, Lampariello and 
Lucidi~\cite{grippo1988global}. Those strategies were then suitably adapted by 
Lucidi and Sciandrone~\cite{lucidi2002global}} to define new globally 
convergent direct-search schemes that make use of an extrapolation step.

\rev{Algorithm \ref{alg:linesearch} provides} an example of exploration 
procedure that performs an extrapolation step once \revn{that} a given trial point 
satisfies the acceptance condition \rev{has been found. Such a step ensures 
improved practical performances by accepting larger stepsizes than those
obtained using the basic strategy of Algorithm \ref{alg:basicDS}.} As it can 
be easily seen by observing the scheme, the extrapolation indeed enables us to 
further explore the direction $\db$ by suitably increasing the stepsize 
\rev{(see Steps 5-7) as long as sufficient decrease is guaranteed.}

\begin{algorithm}[h]
	\caption{Line-search-based DS scheme}
	\label{alg:dse}
	\begin{algorithmic}[1]
		\par\vspace*{0.1cm}
		\STATE \textbf{Inputs:} $\xb_0\in \mathbb{R}^n$, $m \in N$, 
		\revn{$\{\alpha^j_0\}_{j \in [1:m]}$},  $\theta \in (0, 1), \gamma > 1$, 
		a forcing function $\rho: \Rplus \rightarrow \Rplus$.
		\FOR{  $k=0, 1,...$}
		\STATE Select a set of {\rdj p}oll directions 
		$\mathbb{D}_k=\{ \db_k^j\}_{j \in [\revn{1}:m]}$.
		\STATE Set $j(k)=\mod(k,m)$, ${\alpha}_{k + 1}^i = {\alpha}_{k}^i$ 
		for $i \in [\revn{1}:m] \setminus \{j(k)\}$. 
		\STATE Compute $\xb_{k+1}, {\alpha}_{k + 1}^{j(k)}$ with 
		\textbf{Linesearch\_extrapolation}($\xb_k, \db_k^{j(k)},
		\alpha_k^{j(k)}, \theta, \gamma, \rho$).  
		\ENDFOR
  	\par\vspace*{0.1cm}
	\end{algorithmic}
\end{algorithm}

\begin{algorithm}[h]
	\caption{Linesearch\_extrapolation$(\xb, d, \alpha, \theta, \gamma, \rho)$}
	\label{alg:linesearch}
	\begin{algorithmic}[1]
		\par\vspace*{0.1cm}
		\IF{$f(\xb + \rev{\alpha} \db) \geq f(\xb) - \rho(\beta)$} 
		\STATE \textbf{Return} $\xb, \theta \rev{\alpha}$. 
		\ENDIF
		\WHILE{ $f(\xb + \rev{\alpha} \db) < f(\xb) - \rho(\rev{\alpha})$}  
		\STATE \rev{$\alpha \leftarrow \gamma \alpha$}.
		\ENDWHILE
		\STATE \textbf{Return} $(\xb+\frac{\rev{\alpha}}{\gamma} \db,
		\rev{\alpha}/\gamma)$.	
	\end{algorithmic}
\end{algorithm}

\revn{A full 
algorithm that adapts Algorithm~\ref{alg:basicDS} so as to allow for line searches is given in 
Algorithm~\ref{alg:dse}.} 
When using \rev{\revn{line-search techniques}, it becomes necessary to maintain a distinct stepsize} 
for each direction in the poll set $\mathbb{D}_k$, and to ensure some consistency in 
the choice of $\mathbb{D}_k$. \revn{More precisely,
each sequence of directions $\{\bm{d}_k^{j}\}_{k \in \mathbb{N}}$ must satisfy some specific condition like in \cite{lucidi2002global} (e.g., $\{\bm{d}_k^{j}\}_{k \in \mathbb{N}}$ bounded and such that every limit point  $\{ \bm{\bar d} \}_{j \in [1:m]}$ of $\{\bm{d}_k^{j}\}_{j \in [1:m]} $ is a positive spanning set), , and distinct sequences of stepsizes 
$\{\alpha_k^{j}\}_{k \in \mathbb{N}}$ for $j \in [\rev{1}:m]$ must be defined. For the sequence of directions $\{\bm{d}_k^{j}\}_{k \in \mathbb{N}}$,  a natural approach is to use the same direction throughout all iterations. }
% More precisely, if 
% $\mathbb{D}_k = \{\db_k^{j}\}_{j \in [\revn{1}:m]}$ then 
% $\{\db_k^{j}\}_{k \in \mathbb{N}}$ must change in a continuous way with respect 
% to $\{\xb_k\}$ \revn{(one natural approach being to use the same direction throughout all iterations)}, and distinct sequences of stepsizes 
% $\{\alpha_k^{j}\}_{k \in \mathbb{N}}$ for $j \in [\rev{1}:m]$ must be defined. 
Note there is no distinction between successful and unsuccessful steps, 
since every iteration is guaranteed to produce some \revn{point that violates the}  sufficient 
decrease test.

\rev{Provided the function $f$ is smooth \revn{in the sense of Definition~\ref{def:FLsmooth}} and all polling sets correspond to the same 
positive spanning set, one can obtain similar results to that of deterministic 
directional \revn{direct-search methods}~\cite{lucidi2002global}. Thanks to the use of the 
sufficient decrease, any limit point of the sequence generated by 
Algorithm~\ref{alg:dse} is a stationary point.}

\begin{theorem}\label{t:linesearchsmooth}
Suppose that we apply Algorithm~\ref{alg:linesearch} to an $L$-smooth function $f$ 
under Assumptions~{\rdj \ref{ass:L0}, \ref{ass:dirlowup}, \ref{ass:rho}}  
and \ref{ass:cmDk}. Suppose further that sufficient decrease is used. Then,
	\begin{equation}
		\lim_{k \rightarrow \infty} \|\nabla f(\xb_k)\| = 0.
	\end{equation} 
\end{theorem}

\rev{In addition to global convergence results, worst-case complexity bounds 
have recently been derived for line-search techniques in the smooth setting. 
Brilli et al.~\cite{brilli2023worst} report iteration complexity bounds and 
evaluation complexity bounds of $\cO(n\epsilon^{-2})$ and 
$\cO(n^2 \epsilon^{-2})$, respectively, for a particular version of a 
line-search method. Interestingly, the most efficient version in practice 
is endowed with complexity bounds that are worse by a factor of $n$.}

\rev{To establish convergence of Algorithm~\ref{alg:dse} on nonsmooth problems, 
the positive spanning set assumption is replaced by an asymptotic density 
assumption on the polling directions, akin to that made for mesh-based 
techniques. In particular, the choice $\mathbb{D}_k = \{\gb_k\}$, where 
$\gb_k$ is a random direction in the unit sphere, satisfies this assumption, and 
leads to the following result that corresponds to a special case of the analysis 
in Diouane et al.~\cite{diouane2023inexact}.}
\begin{theorem} \label{t:cdir2}
    Suppose that we apply Algorithm~\ref{alg:linesearch} to a Lipschitz 
    continuous function $f$, under Assumptions~\ref{ass:fbound} and 
    \ref{ass:rho}, with sufficient decrease being used. 
    Let $\{\xb_k\}_{k \in \mathcal{K}}$ be a refining subsequence for the algorithm 
    with refined point $\revn{ \hat{\xb}}$ such that $\bbD_k=\{\gb_k\}$ for every 
    $k \in \mathcal{K}$\revn{, where $\mathcal{K}\subset \N$ denotes a subsequence of unsuccessful iterations}. If any subsequence of $\{\bbD_k\}_{k \in \mathcal{K}}$ is dense 
    in the unit sphere, then $\revn{ \hat{\xb}}$ is a Clarke stationary point.
\end{theorem}

\rev{Finally, we comment on the use of non-monotone acceptance rules, an 
important feature of line-search techniques that can significantly enhance 
both robustness and efficiency of those methods, while also applying to 
other directional direct-search techniques. A non-monotone counterpart 
to the sufficient decrease condition can be defined as 
follows~\cite{grippo2015class}.}

\begin{definition}\label{def:nmdecr}
    Let $M$ be a positive integer. 
    A \textit{non-monotone sufficient decrease} condition for 
    Algorithm~\ref{alg:basicDS} or Algorithm~\ref{alg:linesearch} consists 
    \revn{of}
    replacing~\eqref{eq:basicDSdec} by 
    \begin{equation*}
    \label{eq:nmdecr}
            f(\xb_k + \alpha_k \db_k) 
            < f_k - \rho(\alpha_k), 
            \quad 
            f(\xb_k) \le f_k \le \max_{j \in [0,\min(k,M)]} \{f(\xb_{k-j})\}.
    \end{equation*}
\end{definition}
\rev{Grippo and Rinaldi \cite{grippo2015class} incorporate those 
non-monotone strategies into a wide range of algorithms that make use of 
simplex gradients, thus promoting an extensive integration of direct-search 
and model-based approaches. Standard \revn{direct-search techniques have} also been combined 
with the \revn{non-monotone} principle~\cite{gasparini2001nonmonotone}. Other
non-monotone line-search derivative-free algorithms have been propose by 
Diniz-Ehrhardt et al{\rdj .}~\cite{diniz2008derivative} and Garc\'ia-Palomares and 
Rodr\'iguez~\cite{garcia2002new}. In all cases, a global convergence result 
similar to that of Theorems~\ref{th:liminfdirds} or~\ref{t:cdir2} is 
established.}

%%%%%%%%%%%%%%%%%%%%%%%%%%%%%%%%%%%%%%%%%%%%%%%%%%%%%%%%%%%%%%%%%%%%%%%%%%%%%%%%%%%%
\section{Handling noisy and stochastic values}
\label{sec:stoch}
%%%%%%%%%%%%%%%%%%%%%%%%%%%%%%%%%%%%%%%%%%%%%%%%%%%%%%%%%%%%%%%%%%%%%%%%%%%%%%%%%%%%

This section focuses on direct-search methods for unconstrained optimization 
problems where the objective function cannot be evaluated exactly due to noise in 
the calculation. Standard direct-search methods such as those described in 
Section~\ref{sec:stoch} have long been observed to be robust to noise up to a 
certain magnitude associated with the stepsize~\cite{KoLeTo03a}. However, provably handling noise in function evaluations requires {\rdj modifying the basic direct-search framework to establish} (probabilistic) convergence guarantees.
%However, provably handling noise in the function evaluations requires to modify the basic direct-search framework in order to establish (probabilistic) convergence guarantees. 
Our goal in this section is to provide the reader with key approaches 
in this very active area of research.

%
% In Section~\ref{ssec:boundednoise}, we describe existing results in the \rev{special 
% case of bounded noise}. Section~\ref{ssec:sds} describes direct-search methods for 
% \rev{the general noisy setting and discusses in particular the number of noisy 
% evaluations (also called samples) that are used by the method.}

%%%%%%%%%%%%%%%%%%%%%%%%%%%%%%%%%%%%%%%%%%%%%%%%%%%%%%%%%%%%%%%%%%%%%%%%%%%%%%%%%%%%
\subsection{\revn{Direct-search methods with controlled} noise}
\label{ssec:boundednoise}

In this section, we consider {\rdj P}roblem~\eqref{eq:uncpb} under a \revn{controlled} noise model. 
At any point $\xb \in \R^n$, any query for the true function value $f(\xb)$ returns 
a noisy evaluation \revn{$f({\rdj \x},\ubar{\zeta})$,
\footnote{We recall from Section~\ref{sec:algosdef} that we underline quantities \revn{or objects}
that depend on some form of noise, although that noise may be deterministic in 
nature.} 
where $\ubar{\zeta}$ captures the problem noise, and $f({\rdj \x},\ubar{\zeta})$ represents the noisy function values that can be queried by the algorithm.}
\revn{A controlled noise model assumes that 
such that $|f({\rdj \x},\ubar{\zeta})-f(\xb)|$ is uniformly bounded} by a fixed value called  \emph{noise level}. This noise level typically prevents the method from 
converging below a certain accuracy, yet it is often possible to guarantee 
convergence up to that threshold.

The \revn{controlled} noise model has been analyzed in various ways in the direct-search 
literature. \revn{Kelley's implicit filtering method is a directional direct-search algorithm} equipped with global 
convergence guarantees that depend on the noise level~\cite{kelley2011implicit}. 
The thesis of Pauwels follows a similar approach to derive complexity 
guarantees~\cite[Chapter 3]{pauwels2016optimisation}. 
\rev{In line-search techniques,} Lucidi and Sciandrone~\cite{LuSc02b} consider a 
line-search algorithm similar to those of Section~\ref{ssec:linesearch} under 
the assumption that the noise level at every point considered at iteration 
$k$ is bounded by a deterministic value $\Delta_k>0$. Provided the sequence 
$\left\{\tfrac{\Delta_k}{\min_{j \in \revn{\{1,|\bbD_k|\}}} \ubar{\ssize}_{k, j}}\right\}$ 
\revnn{(where $\underline{\alpha}_{k,j}$ denotes the stepsize at iteration $k$ for 
direction $j$)}
converges to $0$ as $k\rightarrow \infty$, convergence to stationarity points 
is established on smooth problems.

\rev{In the bilevel setting, Diouane et al.~\cite{diouane2023inexact} describe 
directional and mesh-based approaches motivated by bilevel optimization to 
handle \revn{controlled} noise of the form 
$|\ubar{f}(\xb) - f(\xb)| \leq \revn{\varepsilon_f}$ for some \revn{noise level} 
$\revn{\varepsilon_f} > 0$. By using a 
lower bound $\ssize_{\min}$ on the stepsize, and assuming that the method 
stops at the first iteration $k$ satisfying $\ubar{\ssize}_k = \ssize_{\min}$, it 
is possible to provide complexity guarantees for the method. In the smooth case, {\rdj denoting by~$\ubar{\xb}_k$ the current iterate,}
one has }
\[
    \n{\nabla f(\ubar{\xb}_k)} 
    = 
    \cO\left(\frac{\revn{\varepsilon_f}}{\ssize_{\min}}+\ssize_{\min}\right)
\]
after at most $\cO(\ssize_{\min}^{-2})$ iterations. \rev{By setting
$\ssize_{\min}=\cO(\revn{\varepsilon_f}^{1/2})$, one guarantees a gradient accuracy of  
order $\cO(\revn{\varepsilon_f}^{1/2})$ in $\cO(\revn{\varepsilon_f}^{-1})$ iterations. In the nonsmooth 
case, the method converges asymptotically to a 
$(\alpha_{\min},\cO(\revn{\varepsilon_f}/\alpha_{\min} + \alpha_{\min}))$-Goldstein
stationary point in the sense of Definition~\ref{def:approxgoldstein}. 
Setting again $\ssize_{\min}=\cO(\revn{\varepsilon_f}^{1/2})$ guarantees convergence to 
$( {\rdj \cO}(\sqrt{\revn{\varepsilon_f}}), {\rdj \cO}(\sqrt{\revn{\varepsilon_f}}))$-Goldstein stationary points.} \revnn{Note that the noise is stochastic in nature, but the uniform bound is deterministic. As a result, the aforementioned complexity results hold for every realization of the method.}

\rev{Finally, we mention another approach that has been used to deal with 
\revn{controlled} noise in stochastic optimization. Berahas et 
al.~\cite{berahas2019derivative,berahas2021global} proposed line-search 
methods based on a relaxed Armijo condition that explicitly involves the noise 
level value. In a smooth setting, these methods can converge to an approximate 
stationary point with gradient norm of order $\cO(\sqrt{\revn{\varepsilon_f}})$ where 
$\revn{\varepsilon_f}>0$ is the noise level. Although these techniques typically employ 
a gradient approximation, a similar approach could be applied in direct 
search by using a forcing function of the form 
$\rho(\underline{\ssize}-\cO(\revn{\varepsilon_f}))$.}

%%%%%%%%%%%%%%%%%%%%%%%%%%%%%%%%%%%%%%%%%%%%%%%%%%%%%%%%%%%%%%%%%%%%%%%%%%%%%%%%%%%%
\subsection{\revn{Direct-search methods} for general random noise}
\label{ssec:sds}

In this section, we consider more general \rev{forms of noise on function values 
that are not necessarily bounded. A typical formulation of the problem in that 
setting is}
\begin{equation}\label{eq:stochpb}
	\underset{\x\in\rn}{\min} f(\x)
	=\Esp_{{\rd \ubar{\zeta}}}\left[f(\x,\ubar{\zeta})\right],
\end{equation}
\revn{and we again assume that the algorithm can only evaluate noisy 
function values of the form $f({\rdj \x},\ubar{\zeta})$.}
\rev{Problem~\eqref{eq:stochpb} has attracted significant attention in the 
derivative-free optimization community beyond \revn{the direct-search paradigm}. 
Under smoothness assumptions on $f$ and $f(\cdot,\ubar{\zeta})$,} line-search 
techniques that employ a stochastic estimation of the gradient $\nabla f(\x)$ have 
been investigated~\cite{berahas2021global,RBSWlong21,miaoSchNeur2021,
paquette2018stochastic}. Interestingly, recent proposals of that form, termed 
``step search'', modify the search direction along with the stepsize, thus 
bringing them closer to direct-search methods in spirit~\cite{RBSWlong21}. 
Extensions of the direct-search schemes reviewed in Section~\ref{sec:stoch} 
have also been proposed to tackle {\rdj P}roblem~\eqref{eq:stochpb}~\cite{AlAuBoLed2019,
audet2019stomads,Ch2012,dzahini2020expected,dzahini2020constrained,DzaWildSubDir24,
rinaldi2023stochastic}. 
Audet et al.~\cite{audet2019stomads} introduced StoMADS, \rev{which we 
view as a stochastic variant of Algorithm~\ref{alg:mads} for the purpose of 
this survey.} At every iteration, this method uses estimates 
\rev{$\Fok\approx f(\ubar{\xb}_k)$ and $\Fsk\approx f(\ubar{\xb}_k+\s)$ that 
are computed using the noisy function $f(\cdot,{\rdj \ubar{\zeta}})$ or function 
values computed for different values of the noise parameter $\ubar{\zeta}$.}
One drawback of using such estimates is that the simple decrease condition 
$\Fsk<\Fok$ does not necessarily imply 
$f(\ubar{\xb}_k+\ubar{\s})<f(\ubar{\xb}_k)$. To circumvent this issue, Audet et 
al.~\cite{audet2019stomads} {\rdj used} so-called $\ef$-accurate estimates, 
inspired by derivative-free trust-region algorithms~\cite{chen2018stochastic} 
and proved that a \emph{sufficient} decrease in such estimates guarantees an 
improvement in $f$. This notion \rev{is stated below for a realization of the 
algorithm at hand (hence all quantities \revn{or objects considered} are deterministic).}

\begin{definition}{\rev{(\cite[Definition~1]{audet2019stomads} and
\cite[Definition~1]{dzahini2020expected}.)}}
\label{def:efacc}
    Consider a realization of Algorithm~\ref{alg:mads} applied to 
    {\rdj P}roblem~\eqref{eq:stochpb}, and let $\xk$ be the incumbent at iteration 
    $k$. Then, for any $\bm{s} \in \R^n$ and any $\ef>0$, a value $f_k^{\bm{s}}$ 
    is called an $\ef$-accurate estimate of $f(\xb+\bm{s})$ for~$\dpl$ if 
	\[
		\abs{f_k^{\bm{s}}-f(\xk+\bm{s})}
		\leq 
		\ef(\dpl)^2.
	\]
\end{definition}
\revn{Similar} to the analysis in Section~\ref{sssec:directionalrandomsmooth}, 
\rev{probabilistic satisfaction of the accuracy property can be considered.}

\begin{definition}
	A sequence $\accolade{\Fok, \Fsk}$ of random estimates produced by the stochastic version of Algorithm~\ref{alg:mads} is called $\beta$-probabilistically $\ef$-accurate with respect to the corresponding sequence $\accolade{\Xk,\Sk,\Dplj}$, if
	\begin{equation}\label{probAccEstDef1}
		\prob{\accolade{\abs{\Fok-f(\Xk)}\leq \ef(\Dplj)^2}\cap \accolade{\abs{\Fsk-f(\Xk+\ubar{\bm{s}_k})}\leq \ef(\Dplj)^2}|\F_{k-1}}\geq \beta.
	\end{equation}
\end{definition}

The convergence analysis of StoMADS uses the following key assumptions on the nature of the stochastic information.

\begin{assumption}\label{StochAssumptStom}
	The sequence $\accolade{\Fok, \Fsk}$ of random estimates 
	\rev{is $\beta$-probabilistically $\ef$-accurate for some $\beta\in (1/2, 1)$. 
	Moreover, for any $k \ge 0$, one has
	\begin{equation}
		\E{\abs{\Fsk-f(\Xk+\Sk)}^2|\F_{k-1}} \le \kappa_f^2(\underline{\ssize}_k^{\revn{p}})^4
		\quad \mbox{and} \quad
		\E{\abs{\Fok-f(\Xk)}^2|\F_{k-1}} \le \kappa_f^2(\underline{\ssize}_k^{\revn{p}})^4
	\end{equation}
	almost surely for some $\kappa_f \ge 0$.}
\end{assumption}

\rev{Random estimates satisfying Assumption~\ref{StochAssumptStom} can easily be 
obtained when the function estimates have bounded variance, i.e. there exists 
$V>0$ such that 
$\mathbb{V}_{{\bl \ubar{\zeta}}}\left[f(\x,\ubar{\zeta})\right]\le V$ for all 
$\x\in\rn$. Under this assumption, one construct{\rdj s} the estimates by averaging 
independent samples as follows:
\begin{equation}
\label{eq:estimreplications}
	\Fok:=\frac{1}{p_k}\sum_{i=1}^{p_k}f(\Xk,{{\bl \ubar{\zeta}}}_i^{\bm{0}})
	\qquad\mbox{and}\qquad 
	\Fsk:=\frac{1}{p_k}\sum_{i=1}^{p_k}f(\Xk+\Sk,{{\bl \ubar{\zeta}}}_i^{\bm{s}}).
\end{equation}
Provided $p_k\geq \frac{V}{\ef^2(1-\sqrt{\beta})(\dpl)^4}$, the 
estimates~\eqref{eq:estimreplications} satisfy Assumption~\ref{StochAssumptStom} 
with $\kappa_f^2 \ge \ef^2(1-\sqrt{\beta})$}.

\rev{Drawing inspiration again on derivative-free trust-region 
methods~\cite{blanchet2016convergence,chen2018stochastic,paquette2018stochastic}}, 
the analysis of StoMADS \rev{provides a condition on the probability $\beta$ for 
guaranteeing convergence.} \revnn{This condition depends on an auxiliary parameter $\nu$ corresponding to the Lyapunov function $\nu(f({\xb}_k)-\min_{\mathbf{x} \in \R^n}f(\xb))+(1-\nu)\{\dpl\}^2$. The requirements on $\nu$ (and thus $\beta$) arise naturally from the analysis, and as such are difficult to draw intuition from. On the other hand, we note that the bounds~\ref{nubetachoiceStoMADS} can  be estimated for specific parameter values~\cite[Remark 4.13]{chen2018stochastic}.}

\begin{theorem}{~\cite[Theorem~1 and Corollary~1]{audet2019stomads}}
\label{convSeriesStom}
	Under Assumptions~\ref{ass:fbound} and~\ref{StochAssumptStom}, if 
	$\nu\in (0,1)$ and $\beta\in (1/2,1)$ satisfy
	\begin{equation}\label{nubetachoiceStoMADS}
		\frac{\nu}{1-\nu}\geq \frac{2({\bl \theta}^{-4}-1)}{\ef(\gamma-2)}
		\qquad\mbox{and}\qquad 
		\frac{\beta}{\sqrt{1-\beta}}\geq\frac{4\nu\kappa_f}{(1-\nu)(1-{\bl \theta^{2}})},
	\end{equation}
	the sequence $\accoladekinN{\Dm}$ of random mesh size parameters satisfies 
	$\sum_{k=0}^{\infty}\Dm<\infty$ almost surely. Consequently, 
	$\underset{k\to+\infty}{\lim}\Dm=0$ almost surely.
\end{theorem}
The last result of Theorem~\ref{convSeriesStom} shows that on an event of 
probability one, the meshes used by StoMADS become infinitely fine. \rev{Such a 
zeroth-order convergence result is stronger than that of Algorithm~\ref{alg:mads} 
(recall Theorem~\ref{zerothOrdMADStheor}), and more generally of mesh-based 
techniques that do not rely upon sufficient decrease. In fact, sufficient 
decrease appears necessary to obtain convergence and complexity results not 
only for \revn{direct-search algorithms}, but also for other derivative-free 
techniques~\cite{chen2018stochastic}.}

\rev{The existence of a convergent refining subsequence was established for 
StoMADS under the assumption that all the iterates generated by the method 
belong to a compact set~\cite[Theorem~2]{audet2019stomads}. Due to possibly 
inaccurate estimates, the algorithm can accept points that increase the 
actual function value, and may thus venture outside the initial level set. 
As a result, Assumption~\ref{ass:L0} cannot be used to establish convergence. 
Still, a compactness assumption yields the desired result.}
\begin{theorem}\label{refConvStoM}
    \rev{Under the assumptions of Theorem~\ref{convSeriesStom}, suppose 
    that all the iterates generated by StoMADS lie in a compact set for 
    all realization. Then, there exists a refining subsequence 
    $\accolade{\Xk}_{k\in \mathcal{K}}$ defined by a random index set $\mathcal{K}$ 
    that converges almost surely to a refined point $\hat{\ubar{\x}}$.}
\end{theorem}

\rev{Finally, the analysis of StoMADS provides a stationarity result for the 
nonsmooth setting, akin to Theorem~\ref{theo:madscdir}.}

\begin{theorem}
\label{th:cvstomadsnonsmooth}
	\rev{Let the assumptions of Theorem~\ref{refConvStoM} hold. Then, there exists 
	an almost-sure event $\mathscr{E}$ such that, conditioned on $\mathscr{E}$, 
	for any refined point $\hat{\ubar{\x}}\in\rn$ and any refining direction 
	$\ubar{\hat{\di}}\in\rn$ for $\hat{\ubar{\x}}$, 
	one has 
	\[
		f^{\circ}\left(\hat{\ubar{\x}}; \ubar{\hat{\di}}\right)
		\geq 0.
	\]
	}
\end{theorem}

\rev{The argument and analysis of StoMADS was adapted to directional 
direct-search methods by Dzahini~\cite{dzahini2020expected}, giving rise to the 
Stochastic Directional \revn{Direct-Search} (SDDS) algorithm, which can be viewed as the 
extension of Algorithm~\ref{alg:basicDS} to handle noisy function values. The SDDS 
framework which, unlike StoMADS, does not require that trial points belong to a 
mesh, relies on a careful update parameter for the stepsize $\alpha_k$, as well 
as sufficient decrease~\cite[Section~2.1]{dzahini2020expected} based on a forcing function of the form $\revnn{ \rho(\alpha)=c\,\alpha^q,\ \alpha>0}$. 
Using martingale-based analysis~\cite{blanchet2016convergence} and assuming 
$\beta$-probabilistically $\ef$-accurate estimates, one can show that the 
expected number of iterations required by SDDS algorithms to drive 
$\normii{\nabla f(\Xk)}$ below a given threshold $\epsilon\in (0,1)$ is of order 
$\cO\left(\epsilon^{-\frac{q}{\min(q-1, 1)}}/(2\beta - 1) \right)$
~\cite[Theorem~4]{dzahini2020expected}.} 
More recently,  Rinaldi et al.~\cite{rinaldi2023stochastic} proved convergence 
of a stochastic \rev{directional} direct-search scheme for nonsmooth objectives. 
\rev{The convergence analysis uses a weak tail bound condition on the noise, a 
weaker assumption than prior results, that still leads to almost-sure convergence.}

\paragraph{\rev{Sample complexity}} 
\rev{Sample complexity refers to bounding the number of evaluations of the noisy 
function estimates (also called samples) required to achieve a desired convergence 
level. To the best of our knowledge, there are no bounds available on the number 
of samples required to achieve a given accuracy for direct-search methods with 
adaptive sampling, despite recent developments in the general stochastic 
optimization setting~\cite{jin2023sample}. On the other hand, bounds on the 
number of samples per iteration have been obtained in the direct-search 
literature. A first bound of that nature was given by Audet et 
al.~\cite{audet2019stomads}, where estimates~\eqref{eq:estimreplications} required 
$\cO({\revn{\underline{\alpha}_k}}^{-4})$ samples in order to obtain convergence 
guarantees. A similar bound in $\cO(\underline{\alpha}_k^{-4})$ can be showed 
for SDDS~\cite{dzahini2020expected}. An improvement to $\cO(\underline{\alpha}_k^{-2q})$ 
samples per iteration was proved by Rinaldi et al.~\cite{rinaldi2023stochastic} 
based on using a forcing function of the form $\rho(\alpha)=c\,\alpha^q$ with 
$c>0$. Assuming a common random number generator can be used to sample the 
function estimates, and that $f(\cdot,\revn{\ubar{\zeta})}$ is Lipschitz continuous 
with a Lipschitz constant that is uniform across all \revn{$\ubar{\zeta}$}, the 
sampling cost improves by a factor of $\alpha_k^2$ at every 
iteration~\cite{rinaldi2023stochastic}.} Another notable result is the 
$\cO((\log T)^{\frac{2}{3}} T^{\frac{2}{3}})$ regret bound with respect to a 
sample budget~$T$ for stochastic direct-search applied to smooth strongly convex 
objectives proved by Achddou et al.~\cite{achddou2022regret}. The same authors 
introduce a novel technique called \emph{sequential sampling}, that allows for 
fewer samples per iteration in cases where the sufficient decrease condition is 
satisfied or violated with a certain margin. 

%%%%%%%%%%%%%%%%%%%%%%%%%%%%%%%%%%%%%%%%%%%%%%%%%%%%%%%%%%%%%%%%%%%%%%%%%%%%%%%%%%%%
\section{Handling constraints}
\label{sec:cons}
%%%%%%%%%%%%%%%%%%%%%%%%%%%%%%%%%%%%%%%%%%%%%%%%%%%%%%%%%%%%%%%%%%%%%%%%%%%%%%%%%%%%

In this section, we consider direct-search methods designed for \emph{constrained} 
optimization problems of the form
\begin{equation}
\label{eq:constrained}
	\min_{\xb \in \R^n} f(\xb) 
	\qquad \mbox{s.t.} \quad 
	\xb \in \Omega, 
\end{equation}
where \rev{the feasible set $\Omega$ is a subset of $\R^n$ that encodes constraints 
on the problem variables}. \revn{As in the rest of the paper, information about the objective function is only available through function evaluations. As for the constraint set, access to derivative information is common for simple constraints, but is not necessary to design a 
direct-search scheme. In fact, constraints in derivative-free information are subject to a careful 
classification}, according to the information they 
provide regarding feasibility \rev{or lack thereof}~\cite{le2023taxonomy}. In the 
context of this survey, we \revn{classify a constraint} as \emph{relaxable} or 
\emph{unrelaxable} depending on whether \revn{the objective function and all other constraints can be queried at points that do not satisfy the first constraint.} We \revn{also} \rev{distinguish 
\emph{quantifiable} and \emph{non-quantifiable} constraints: the former 
corresponds to constraints that produce a numerical value, while the latter} 
can be viewed as a binary certificate of constraint satisfaction.

\rev{The first part of this section is concerned with feasible direct-search 
methods, that aim at preserving feasibility of the iterates throughout the 
optimization process. Section~\ref{ssec:feasDS} surveys theoretical guarantees 
that have been obtained in the presence of linear constraints, manifold 
constraints or non-relaxable constraints. We then move to infeasible 
direct-search methods, for which feasibility can only be guaranteed 
asymptotically, in Section~\ref{ssec:infeasDS}.}

%%%%%%%%%%%%%%%%%%%%%%%%%%%%%%%%%%%%%%%%%%%%%%%%%%%%%%%%%%%%%%%%%%%%%%%%%%%%%%%%%%%%
\subsection{Feasible methods}
\label{ssec:feasDS}

In this section, we consider feasible direct-search algorithms, that are 
particularly suited for unrelaxable constraints. 
\rev{Algorithm~\ref{alg:feasibleDS} describes a basic variation on 
Algorithm~\ref{alg:basicDS} that ensures the objective to only be evaluated at
feasible points.}

\rev{In Section~\ref{sssec:boundslinds}, we discuss the case of linear constraints, 
that has been tackled by the main classes of direct-search algorithms.}
Sections~\ref{sssec:riemann} and~\ref{sssec:extremebarrier} survey other feasible 
iterate strategies, based on Riemannian manifolds and extreme barrier, 
respectively.

\begin{algorithm}[h]
	\caption{Feasible direct-search scheme}
	\label{alg:feasibleDS}
	\begin{algorithmic}[1]
        \par\vspace*{0.1cm}
		\STATE \textbf{Inputs:} Feasible point $\x_0\in\Omega$, 
		initial stepsize $\ssize_0> 0$, 
        stepsize update parameters $0 < \theta < 1 \le \gamma$, 
        forcing function $\rho: \Rplus \rightarrow \Rplus$.
		\FOR{$k=0,1,2,\ldots$}
			\STATE Select a set $\bbD_k$ of poll directions. 
		      \IF{\rev{$\xk+\bm{s} \in \Omega$ and 
		      $f(\xk+\bm{s}) < f(\xk) - \rev{\rho(\ssize_k)}$ hold for some 
		      $\bm{s} \in \accolade{\ssize_k\db:\bm{d} \in \bbD_k}$}}
                \STATE Declare the iteration as successful,
                set $\xkun=\xk+\bm{s}$ and $\ssize_{k+1}= \gamma \ssize_k$.
            \ELSE
			    \STATE Declare the iteration as unsuccessful, set $\xkun=\xk$ and 
                $\ssize_{k+1}=\theta \ssize_k$.
            \ENDIF
		\ENDFOR
	\end{algorithmic}
\end{algorithm}

%%%%%%%%%%%%%%%%%%%%%%%%%%%%%%%%%%%%%%%%%%%%%%%%%%%%%%%%%%%%%%%%%%%%%%%%%%%%%%%%%%%%
\subsubsection{Linear constraints}
\label{sssec:boundslinds}

\rev{Linear constraints, and in particular bound constraints, are the simplest kind 
of constraints in optimization. In \revn{direct-search techniques}, linear 
constraints can be easily handled by a careful choice of polling 
directions~\cite{KoLeTo03a}. We present below the main tools and results for 
this purpose in the context of directional \revn{direct-search schemes}, then 
discuss related results for mesh-based and line-search techniques.}

\rev{Suppose that the feasible set $\Omega$ of \revn{Problem}~\eqref{eq:constrained} is 
polyhedral and described by linear inequalities, i.e.{\bll,}}
\begin{equation} \label{eq:linearly_constrained}
	\Omega = \{\xb \in \mathbb{R}^n \ | \ \bm{A}\xb \leq \bm{b}\}
\end{equation}
for $\bm{A} \in \mathbb{R}^{l \times n}$ and $\bm{b} \in \mathbb{R}^l$. 
Without loss of generality, we assume \rev{that the rows of $\bm{A}$ are 
unit vectors. Convergence results for Algorithm~\ref{alg:feasibleDS} in that 
setting aim at proving that the quantity $\chi(\xb_k)$ goes to zero, where 
\begin{equation}
\label{eq:chi}
	\forall \xb \in \Omega,
	\chi(\xb)
	=
	\max_{\substack{\xb+\db \in \Omega \\ \|\db\| \le 1}} 
	\scal{\db}{[-\nabla f(\xb)]}.
\end{equation}
The measure $\chi(\xb)$ is always nonnegative, and equals zero if and only if 
$\xb$ is a first-order KKT stationary point of \revn{Problem}~\cite{KoLeTo03a}.}

The analysis of \revn{direct-search methods} in that setting relies heavily 
on the notion of approximate tangent and normal cones\revn{,that involve 
both a feasible point and a tentative stepsize}. 
\revn{
\begin{definition}\label{DefNormTan}
Given $\xb \in \Omega$ and $\alpha>0$, the \emph{approximate normal cone} $N_{\Omega}(\xb,\alpha)$ is defined by
\begin{equation*}
\label{eq:approxnormal}
	N_{\Omega}(\xb,\alpha) 
	= 
	\mbox{pspan}\left(\left\{\ab_i\ |\ b_i-\scal{\ab_i}{\xb} \le \alpha\right\}\right),
\end{equation*}
where $\mbox{pspan}(\mathcal{A})$ denotes the set of nonnegative 
linear combinations of vectors in $\mathcal{A}$.
The \emph{approximate normal cone} is then defined as 
$T_{\Omega}(\xb,\alpha)=N_{\Omega}(\xb,\alpha)^{\circ}$.
\end{definition}
}
It follows from \rev{Definition~\ref{DefNormTan} that a displacement of length $\alpha$ along any direction in the 
approximate tangent cone preserves feasibility, while the normal cone is generated by normals corresponding to constraints that can be violated with such a displacement.} A natural strategy for defining 
polling directions in Algorithm~\ref{alg:feasibleDS} then consists {\rdj of} using 
generators of the approximate tangent cone at the current 
iterate~\cite[Proposition 2.1]{KoLeTo2006}.

\rev{Indeed, the cosine measure defined in Section~\ref{sec:algosdef} can be 
generalized to the linearly constrained setting~\cite{KoLeTo03a,GrRoViZh2019} \revn{as follows}. 
\revn{
\begin{definition}
Given $\xb \in \Omega$ and $\alpha>0$, let 
$\pi_{T_{\Omega}(\xb,\alpha)}[\cdot]$ denote the orthogonal projection onto 
the approximate tangent cone $T_{\Omega}(\xb,\alpha)$. Given $\xb \in \Omega$, $\alpha>0$, and a direction set $\bbD$, the cosine measure generalized to linear constraints is defined by
\begin{equation}
\label{eq:cmTxalpha}
	\textnormal{cm}_{T_{\Omega}(\xb,\alpha)}(\mathbb{D}) 
	=
	\inf_{\substack{\vi \in \R^n\\ \|\pi_{T_{\Omega}(\xb,\alpha)}[\vi]\| \neq 0}} 
	\max_{\db \in \bbD} \frac{\scal{\db}{\vi}}{\|\db\|\|\vi\|}.
\end{equation}
\end{definition}
}
%where we recall that $\pi_{T(\xb,\alpha)}[\cdot]$ denotes the orthogonal projection. 
When $\bbD$ contains generators of 
$T(\xb,\alpha)$, it can be shown that this quantity is 
positive~\cite{KoLeTo2006}, and thus such directions 
guarantee feasible descent, i.e. preserve feasibility and yield descent from 
$\xb$ with a stepsize $\alpha$ as long as 
$\pi_{T_{\Omega}(\xb,\alpha)}\left[-\nabla f(\xb)\right]>0$. Although the latter 
quantity is not a good measure of stationarity, it bears a natural connection 
with the measure $\chi$, described below.}

\begin{lemma}{(\cite[Proposition B.2]{KoLeTo2006})}
\label{l:chiconnect}
	\rev{
	Suppose that $\Omega$ is defined via the polyhedral description~\eqref{eq:linearly_constrained}, and that there exists $B_g>0$ 
	such that $\|\nabla f(\xb)\| \le B_g$ for any $\xb \in \Omega$. Then, for 
	any $\xb \in \Omega$ and any $\alpha>0$, there exists $\eta>0$ such that
	\begin{equation} \label{eq:chibound}
		\chi(\xb) \leq \pi_{T_{\Omega}(\xb,\alpha)}\left[-\nabla f(\xb)\right] 
		+ \frac{\alpha}{\eta}B_g.
	\end{equation} 
	}
\end{lemma}
\rev{
Note that $\eta$ can be computed a priori by considering all possible 
normal cones at $\xb$~\cite{GrRoViZh2019}.}

\rev{Lemma~\ref{l:chiconnect} is instrumental to deriving global convergence 
guarantees for Algorithm~\ref{alg:feasibleDS}. Assuming that generators 
for the approximate tangent cone are used, Kolda et al.~\cite{KoLeTo03a} 
established a global convergence result of the form 
$\liminf_{k \rightarrow \infty} \chi(\xb_k)=0$. A complexity bound was 
later established by Gratton et al.~\cite{GrRoViZh2019}, and is summarized 
below.}

\begin{theorem}{\rev{(\cite[Theorem 4.1]{GrRoViZh2019})}}
\label{th:classic_conv}
	\rev{
	Suppose that $\Omega$ is defined via the polyhedral 
	description~\eqref{eq:linearly_constrained}, and that there exists $B_g>0$ 
	such that $\|\nabla f(\xb)\| \le B_g$ for any $\xb \in \Omega$. Suppose 
	further that $f$ satisfies Assumptions~\ref{ass:L0} and~\ref{ass:fbound}. 
	Finally, suppose that Algorithm~\ref{alg:feasibleDS} is applied using 
	$\rho(\alpha) = c\,\alpha^2$ with $c>0$ and choosing $\bbD_k$ as a set 
	of unit norm generators for $T(\xb_k, \alpha_k)$ at every iteration 
	with $|\bbD_k| \le m$ and 
	$\textnormal{cm}_{T_{\Omega}(\xb_k,\alpha_k)}(\bbD_k) \ge \tau>0$. Then, for 
	any $\epsilon>0$, the method reaches an iterate such that 
	$\chi(\xb_k) \le \epsilon$ in at most
	\begin{equation}
	\label{eq:classic_conv}
		\cO(\tau^{-2}\epsilon^{-2}) \ \mbox{iterations} 
		\quad \mbox{and} \quad
		\cO(m\tau^{-2}\epsilon^{-2}) \ \mbox{function evaluations}.
	\end{equation}
	}
\end{theorem}
\rev{Although the bounds can be shown to hold for finite values of $m$ 
and positive values of $\tau$, precise values for these quantities are 
only available in the special case of bound constraints and linear 
equality constraints. \revn{When only bound constraints are enforced, using coordinate vectors and their negatives as poll directions yields $\tau=1/\sqrt{n}$ and $m=2n$. When the constraint set is defined using $1 \le n' \le n$ linearly independent equality constraints, the tangent cone is a linear subspace of $\R^n$. Using orthonormal basis vectors for that subspace and their negatives as poll directions yields $\tau=1/\sqrt{n'}$ and $m=2n'$~\cite{GrRoViZh2019}.} We note that a generalization of 
Theorem \ref{th:classic_conv} through probabilistic analysis was described in 
Gratton et al.~\cite{GrRoViZh2019}, \revn{showing that high-probability complexity results could also be derived akin to those of Theorem~\ref{t:ncomplexity}.}
Although the random polling sets to be used must be restricted to 
preserve feasibility, it is not clear that the complexity improves 
compared to the deterministic setting, except in the context of 
linear equality constraints. Nevertheless, by decomposing all tangent cones as the 
sum of pointed cones and subspaces, random directions can be drawn in the former, 
leading to a possible improvement of the number of directions used at every 
iteration.}

\rev{We note that linear constraints have also been investigated in the context of 
mesh-based algorithms, with more emphasis on handling linear inequality constraints 
and bounds to obtain convergence results~\cite{abramson2008pattern,price2003frames}. 
Although those result apply in both a smooth and a nonsmooth setting, no 
complexity guarantees are available.}
\rev{Furthermore, in a line-search setting, we note that the feasible direction 
framework based on approximate tangent cones was used in the context of bound 
constraints in Lucidi and Sciandrone~\cite{LuSc02b} to establish convergence 
guarantees. The idea was then extended to general algebraic unrelaxable smooth 
constraints through the use of projections, leading again to global convergence 
guarantees~\cite{lucidi2002objective}.}
\revn{Finally, Cristofari and Rinaldi~\cite{cristofari2021derivative} defined 
 feasible direct-search methods for problems with structured feasible set (i.e., 
$\Omega = \textnormal{conv}(A)$ with $A=\{v_1, \dots, v_r\}, \ v_i \in \R^n$)
that use directions pointing towards/away from the vertices of the feasible set, and are somehow related to \emph{zeroth-order} Frank--Wolfe methods ( see, e.g., \cite{bomze2021frank} and references therein for further details on the Frank--Wolfe method).}

%%%%%%%%%%%%%%%%%%%%%%%%%%%%%%%%%%%%%%%%%%%%%%%%%%%%%%%%%%%%%%%%%%%%%%%%%%%%%%%%%%%%
\subsubsection{Riemannian \revn{direct-search} schemes}
\label{sssec:riemann}

\rev{In this section, we discuss an emerging trend in feasible methods 
related to Riemannian optimization, a well-established subfield of nonlinear 
optimization
\revn{that is concerned with finding the best value of a function over a smooth manifold}
}~\cite{boumal2023introduction}. \rev{Interestingly, 
the idea of extending direct-search methods to the Riemannian setting started 
early in the development of numerical Riemannian optimization.} The first 
direct-search proposal for Riemannian optimization problems is due to 
Dreisigmeyer~\cite{dreisigmeyer2006equality}, and consists {\rdj of} an extension of 
\rev{mesh-based algorithm} to Riemannian manifolds. 
\rev{The special case of reductive homogeneous spaces, that includes several 
matrix manifolds, is discussed in a follow-up work, where both deterministic 
and probabilistic directional methods are 
proposed~\cite{dreisigmeyer2018direct}. Those references focus on properly 
defining the algorithms rather than on establishing convergence guarantees.}

\rev{More recently, Kungurtsev et al.~\cite{kungurtsev2022retraction} describe 
basic direct-search schemes of directional and line-search types for minimizing 
smooth and nonsmooth objectives on Riemannian manifolds. Global convergence 
properties are established, that match those of the unconstrained setting.} \revn{More precisely, counterparts to 
Theorem~\ref{th:lipgglobal} and~\ref{t:cdir} are established for 
directional variants on smooth and Lipschitz continuous objectives, 
respectively, while counterparts to Theorem~\ref{t:linesearchsmooth} and 
~\ref{t:cdir2} are obtained for variants using line-search extrapolation.
We note, however, that complexity results have yet to be derived 
for Riemannian \revn{direct-search schemes}.}

%%%%%%%%%%%%%%%%%%%%%%%%%%%%%%%%%%%%%%%%%%%%%%%%%%%%%%%%%%%%%%%%%%%%%%%%%%%%%%%%%%%%
\subsubsection{Extreme barrier}
\label{sssec:extremebarrier}

The extreme barrier approach is \rev{arguably} the most general way to handle 
constraints in direct-search methods so as to enforce feasibility. It consists {\rdj of} 
setting the objective equal to $+ \infty$ outside the feasible set. The first proof 
of convergence (to Clarke stationary points) of an extreme barrier approach was 
done for mesh adaptive \revn{direct-search algorithms}~\cite{AuDe2006}. \rev{Consequently, most of 
the theory in Section~\ref{ssec:meshbased} extends to the case of a constrained 
set handled through extreme barrier. We present below the main ingredients to 
derive convergence results in the context of Algorithm~\ref{alg:mads} below, but 
note that similar arguments can be applied to any direct-search method with a 
dense set of poll directions.}
\rev{Akin to Section~\ref{sssec:boundslinds}, we start by defining suitable 
cones that characterize stationarity.}
 
\begin{definition}
\label{def:hypertangent}
	\rev{The hypertangent cone $T^H_{\Omega}(\xb)$ to $\Omega$ at 
	$\xb \in \Omega$ is the set of vectors $\db \in \R^n$ such that there 
	exist $\varepsilon > 0$, $\yb \in \Omega \cap B_{\varepsilon}(\xb)$ and 
	$v \in B_{\varepsilon}(\db)$ satisfying $\yb + t \vb \in \Omega$ for 
	any$t \in (0,\varepsilon)$.}
\end{definition}

\begin{definition}
\label{def:clarketgtcone}
	\rev{The \emph{Clarke tangent cone} $T_{\Omega}(\xb)$ to $\Omega$ at 
	$\xb \in \Omega$ is the set of vectors $\db \in \R^n$ such that 
	there exist sequences $\{t_k\}$, $\{\xb_k\}$, $\{\db_k\}$ 
	satisfying $t_k \rightarrow 0^+$, $\xb_k \rightarrow \xb$,
	$\db_k \rightarrow \db$ and $\xb_k + t_k \db_k \in \Omega$ 
	for every $k$.}

    The set $\Omega$ is called \emph{regular} at $\xb$ if 
    $T_{\Omega}(\xb)$ is the closure of $T^H_{\Omega}(\xb)$. 
\end{definition}

\rev{The use of hypertangent cones allows for generalizing 
Lemma~\ref{l:cvrefineddir} to the extreme barrier setting.}

\begin{lemma} 
\label{l:cvrefineddirbarrier}
    Suppose that $f$ is Lipschitz continuous, and consider that 
    Algorithm~\ref{alg:mads} is applied to~\eqref{eq:constrained} using  
    an extreme barrier. Let $\mathcal{K} \subset \N$ be an index subsequence such that 
    $\{\xb_k\}_{k \in \mathcal{K}}$ is a refining subsequence with refined point $\revn{ \hat{\xb}}$. 
    Then, for any refining direction $\db \in T_{\Omega}^H(\revn{ \hat{\xb}})$, 
	\begin{equation}
    \label{eq:cvrefineddirbarrier}
		f^\circ(\revn{ \hat{\xb}}; \db) \geq 0 \, .
	\end{equation}
\end{lemma}

\rev{Similarly, using the Clarke tangent cone together with the result 
of Lemma~\ref{l:cvrefineddirbarrier} leads to the following generalization 
of Theorem~\ref{t:cdir}.}

\begin{theorem}
\label{th:constrained}
	\rev{Under the assumptions of Lemma~\ref{l:cvrefineddirbarrier}, let 
	$\mathcal{K}$ be an index set corresponding to unsuccessful iterations for 
	Algorithm~\ref{alg:mads} applied with extreme barrier. Suppose further that 
	$\lim_{k \in \mathcal{K}} \alpha_k = 0$, $\xb_k \rightarrow \revn{\hat{\xb}}$, and that 
	for every direction $\db \in T^H_{\Omega}(\revn{ \hat{\xb}})$ there exists 
	$\mathcal{L} \subset \mathcal{K}$ such that 
	\revn{$\lim_{k \in \mathcal{L}} \frac{\db_{k}}{\n{\db_{k}}} 
	= \frac{\db}{\n{\db}}$ with 
	$\db_{k} \in \bbD_{k}$}. Finally, suppose that the set $\Omega$ is 
	regular at $\revn{\hat{\xb}}$. Then the point $\revn{\hat{\xb}}$ is Clarke
	stationary for \revn{Problem}~\eqref{eq:constrained}.}
\end{theorem}

The extreme barrier approach thus provides a straightforward way of extending 
\revn{a direct-search method to constraint sets while preserving convergence 
guarantees.} \rev{On the other hand, when the constraints 
are quantifiable, using a progressive barrier prevents from getting any information 
regarding constraint violation. Infeasible approaches have been developed with that 
issue in mind, and we review them in the next section.}

%%%%%%%%%%%%%%%%%%%%%%%%%%%%%%%%%%%%%%%%%%%%%%%%%%%%%%%%%%%%%%%%%%%%%%%%%%%%%%%%%%%%
\subsection{Infeasible methods}
\label{ssec:infeasDS}

This section is concerned with direct-search schemes that allow for evaluating the 
objective outside of the feasible set. Such methods are termed \emph{infeasible}, 
and apply to problems with quantitative relaxable constraints of the form
\begin{equation}
	\label{eq:quantitative}
	\begin{aligned}
		& \min_{\xb \in \R^n} f(\xb) 
		\qquad \mbox{s.t.} \quad \bm{g}(\xb) \le \mathbf{0} \, ,
	\end{aligned}
\end{equation}
where $\bm{g}: \mathbb{R}^n \rightarrow \mathbb{R}^l$. Thanks to the quantitative 
nature of the constraints, it is possible to define an infeasibility measure 
$h(\cdot)$ that quantifies the distance to infeasibility. Typical choices for 
such measures are
\begin{equation}
	\label{eq:infeasibility}
	\rev{h}(\xb) = \sum_{i = 1}^l \max\{0, g_i(\xb)\},
\end{equation}
\revn{which is typically nonsmooth,} and
\begin{equation}
	\label{eq:infeasibility2}
	\rev{h}(\xb) = \sum_{i = 1}^l \max\{0, g_i(\xb)^2\} 
\end{equation}
\revn{which is a smooth penalty function} for smooth functions $g_i$.

\rev{The theory for such \emph{progressive barrier} functions was developed 
in Audet and Dennis~\cite{audet2009progressive}. A two-phase approach for 
solving \revn{Problem}~\eqref{eq:quantitative} using $h$ consists {\rdj of} applying a 
suitable optimization method to the problem 
\begin{equation}
\label{eq:minhpb}
	\min_{\xb \in \R^n} h(\xb)
\end{equation}
in order to find a feasible point. Once such a point is obtained, the original \revn{Problem}~\eqref{eq:quantitative} can be tackled using the extreme barrier approach of Section~\ref{sssec:extremebarrier}.}

\rev{When the derivatives of $\bm{g}$ are not available, one may consider 
using Algorithm~\ref{alg:basicDS} or any of the methods discussed in 
Section~\ref{sec:unc} in the first phase. To this end, the sufficient 
decrease condition~\eqref{eq:basicDSdec} must be modified so as to accept 
any feasible point that is encountered by the algorithm. However, note that the
local nature of direct-search methods only guarantees convergence to a 
stationary point of the barrier function $h$. Still, provided the iterates 
converge towards the feasible set, a feasible point can be found in 
finitely many iterations. In that case, convergence results follow from that 
of the extreme barrier approach, and can be obtained for both progressive 
barrier functions~\eqref{eq:infeasibility} and~\eqref{eq:infeasibility2}.}

\rev{Since the two-phase approach ultimately reduces to an extreme barrier 
approach, it suffers from the same shortcoming in that information at 
infeasible points eventually stops being used. The next three sections describe 
algorithms that exploit infeasibility information.} \revnn{Note that these variants are inspired by derivative-based constrained optimization techniques and, as such, possess the same advantages and drawbacks that their derivative-based counterparts. Thus, rather than formally stating their convergence results, we provide a summary of the key guarantees achievable by each algorithmic strategy, drawing from the main definitions and vocabulary of Section~\ref{ssec:feasDS}.}

%%%%%%%%%%%%%%%%%%%%%%%%%%%%%%%%%%%%%%%%%%%%%%%%%%%%%%%%%%%%%%%%%%%%%%%%%%%%%%%%%%%%
\subsubsection{Filter-based methods}
\label{sssec:filter}

\rev{A filter approach consists {\rdj of} a biobjective view of the constrained 
optimization \revn{Problem}~\eqref{eq:quantitative}, where the two objectives are the 
original function $f$ and the progressive barrier $h$. In a filter-based algorithm, 
a set of ``incumbent'' points $\revn{\mathbb{S}}_k$ is maintained throughout the algorithm, and 
the corresponding values of $f$ and $h$ are stored. At every iteration, $\revn{\mathbb{S}}_k$ is 
updated by adding new points and discarding points dominated by new additions, 
where $\xb$ is said to be dominated by $\yb$ (written $\xb \prec \yb$) if 
\[
	\left[ f(\xb) < f(\yb)\ \mbox{and}\ h(\xb) \le h(\yb)\right]
	\quad \mbox{or} \quad
	\left[ f(\xb) \le f(\yb)\ \mbox{and}\ h(\xb) < h(\yb)\right].
\]
With this definition, the set $\revn{\mathbb{S}}_k$ may consist of both feasible and infeasible 
points, the latter corresponding to better objective values.}

\rev{The first direct-search method employing a filter is a mesh-based algorithm 
proposed by Audet and Dennis~\cite{audet2004pattern}, but that method was not 
proven to converge towards stationary points for the original problem. The 
mesh-based algorithm of Dennis et al.~\cite{DePrCo04} possess such a guarantee, 
provided both the constraint functions and the objective are differentiable 
(with approximate gradients available for the constraints). The progressive 
barrier of Audet and Dennis~\cite{audet2009progressive,AuDeLe10} combines 
\revn{the mesh adaptive direct-search paradigm} with a filter based on a progressive barrier 
function. Algorithm~\ref{alg:pbDS} describes an analogous direct-search 
method based on Algorithm~\ref{alg:basicDS}. Note that this choice is made for 
simplicity, and that one could have presented a variant on 
Algorithm~\ref{alg:mads} for the same purpose.}

\begin{algorithm}[h]
	\caption{\rev{Progressive barrier \revn{direct-search method}}}
	\label{alg:pbDS}
	\begin{algorithmic}[1]
		\par\vspace*{0.1cm}
		\STATE \textbf{Inputs:} \rev{$\revn{\emptyset \neq \mathbb{S}}_0 \subset \R^n$, $\alpha_0> 0$, 
		$0 < \theta < 1 \le \gamma$, $\rho: \Rplus \rightarrow \Rplus$, 
		$h_0 > 0$.}
		\FOR{$k=0,1,2,\ldots$}		
			\STATE \rev{Select a set $\mathbb{D}_k$ of poll directions. 
			Select \revn{a} feasible point $\xb_k^F \in \revn{\mathbb{S}}_k$ if such a point 
			exists, and \revn{an} infeasible point $\xb_k^I \in \revn{\mathbb{S}}_k$ if such a point 
			exists.}
			\IF{\rev{ $\yb + \sba \prec_{\rho(\alpha_k)}  \yb$ holds for \revn{either $\yb=\xb_k^F$ or $\yb=\xb_k^I$} and 
			$\sba \in \{\revnn{\alpha_k} \db \ | \ \db \in \bbD_k \}$}}
			\STATE Declare the iteration as successful, add $\yb + \sba$ to 
			$\revn{\mathbb{S}}_k$, set $h_{k + 1} = h_k$, and $\alpha_{k + 1} = \gamma \alpha_k$.
		\ELSE
			\IF {\revn{$\xb_k^I$ was selected and there exists $\bm{s}\in \accolade{\alpha_k\bm{d}:\bm{d} \in \mathbb{D}_k}$ such that $h(\xk^I + \bm{s}) < h_k - \rho(\alpha_k)$ and  
			$f(\xk^I + \bm{s}) > f(\xk^I)$
			}}
				\STATE Declare the iteration as improving, add 
				$\rev{\xb_k^I} + \bm{s}$ to $\revn{\mathbb{S}}_k$, and \rev{choose 
				$h_{k + 1} \le h_{k} - \rho(\alpha_k)$.}\\
				Set $\alpha_{k + 1} = \alpha_k$.  
			\ELSE 
				\STATE Declare the iteration unsuccessful, set 
				$\revn{\mathbb{S}}_{k + 1} = \revn{\mathbb{S}}_k$, $h_{k + 1} = h_k$ and 
				$\alpha_{k + 1} = \theta \alpha_k$.
			\ENDIF
		\ENDIF
		\STATE Delete all the dominated points from $\revn{\mathbb{S}}_{k + 1}$.
		\ENDFOR
	\end{algorithmic}
\end{algorithm}

\rev{At every iteration, Algorithm seeks improvement with respect to 
either a feasible point (if one has been found) or an infeasible point. 
A successful iteration corresponds to finding a point that dominates one 
of the two points. Note that Algorithm~\ref{alg:pbDS} uses special dominance 
relations of the form $\xb \prec_{\varrho} \yb$ for some $\varrho>0$, that 
correspond to $\xb \prec \yb$ and $f(\xb) < f(\yb) - \varrho$ or 
$h(\xb) < h(\yb)- \varrho$. If one dominant point is found, this point 
is added to the set $\revn{\mathbb{S}}_k$. Otherwise, one checks whether an infeasible 
point with improved barrier value can be found, even if it increases 
the objective value, and such a point is added to $\revn{\mathbb{S}}_k$ if it exists. As 
a result, the method tries to improve both objectives or puts the 
emphasis on constraint violation.}

\rev{In a mesh-based setting, Audet and Dennis~\cite{audet2009progressive} show 
convergence guarantees for refining sequences of the algorithm. If such a sequence 
consists of feasible points, convergence to a (Clarke) stationary point can be 
established under density of the refining directions. If the sequence consists of 
infeasible points, only convergence to a stationary point for 
\revn{Problem}~\eqref{eq:minhpb} \revn{is guaranteed}.}

%\begin{proposition} 
%\label{prop:filteralphak0}
%	Consider a sequence generated by Algorithm \ref{alg:pbDS}. Assume that, for some $\bar{f} > 0$, $f(\yb) \leq \bar{f}$ for every $\yb \in S_k$. Then we have $\lim \alpha_k = 0$.
%\end{proposition}
%\begin{proof}
%	Assume, by contradiction, that we have an infinite number of improving or successful steps with $\alpha_k ~>~\varepsilon$, for some $\varepsilon > 0$. First, since $\{h_k\}$ is non increasing and strictly decreases by $\rho(\alpha_k)$ at every improving step, there cannot be an infinite number of improving steps with $\alpha_k > \varepsilon$. Thus we must necessarily have an infinite number of successful steps with $\alpha_k \geq \varepsilon$. Now, on the one hand $ \bar{S} = \cup_k \{ (\lfloor f(\yb)/\rho(\varepsilon) \rfloor, \lfloor h(\yb)/ \rho(\varepsilon) \rfloor) \}_{\yb \in S_k}$ must be bounded because it is a subset of $[0:\lfloor \bar{f}/\rho(\varepsilon) \rfloor] \times [0:\lfloor h_0/\rho(\varepsilon) \rfloor]$. On the other hand, it is easy to check that due to the $\rho(\alpha_k)$ majorization condition at every successful iteration with $\alpha_k \geq \varepsilon$ a new point is added to $\bar{S}$, a contradiction.
%\end{proof}
%
%Thanks to Proposition \ref{prop:filteralphak0}, the results of Theorem \ref{th:twophase} can be applied directly to Algorithm \ref{alg:pbDS}.  

%%%%%%%%%%%%%%%%%%%%%%%%%%%%%%%%%%%%%%%%%%%%%%%%%%%%%%%%%%%%%%%%%%%%%%%%%%%%%%%%%%%%
\subsubsection{\rev{Restoration-based} methods}
\label{sssec:restoration}

\rev{Restoration-based approaches rely on alternating steps \rev{that} improve the 
objective with so-called restoration steps, \revn{which focus}  on improving the measure 
of infeasibility $h$. The approach of Bueno et al.   \cite{bueno2013inexact} handles 
differentiable constraints, and alternates a gradient-based restoration phase with 
direct-search steps on a problem with linearized versions of the constraints. In 
this smooth setting, convergence to a KKT point for the original 
\revn{Problem}~\eqref{eq:quantitative} can be established.}

\rev{The method of Gratton and Vicente~\cite{gratton2014merit} considers two 
conditions to accept a point at iteration $k$, namely}
\begin{equation}
	\label{eq:restoration}
	h(\xb_k) - h(\xb_k + \alpha_k \db_k) > \rho(\alpha_k) 
	\quad \text{and} \quad 
	h(\xb_k) > C \rho(\alpha_k)
\end{equation}
for some $C > 1$, and 
\begin{equation}
	\label{eq:improvement}
	M(\xb_k, \mu) - M(\xb_k + \alpha_k \db_k, \mu) > \rho(\alpha_k)
\end{equation}
for some $\mu > 0$, \rev{where $M(\xb, \mu) = f(\xb) + \mu h(\xb)$ is referred to 
as a merit function}. Within the restoration phase, only~\eqref{eq:restoration} is 
considered \rev{to accept a new point}. \rev{Convergence results for this 
algorithm are established using the nonsmooth barrier~\eqref{eq:infeasibility}. For 
nonsmooth objective, stationarity of refining points is established with a similar 
distinction than in the filter-based setting of Audet and 
Dennis~\cite{audet2009progressive}. In addition, results are specialized to the 
case of smooth objectives, thereby avoiding to require density of the refining 
directions.}

%While the methodological and theoretical results presented by Gratton and Vicente~\cite{gratton2014merit} are beyond the scope of this survey, we show in the next proposition how a convergence result analogous to the one of two phase methods can be obtained by considering suitable sequences of unsuccessful steps. 
%\begin{proposition}
%	\label{prop:restoration}
%	Assume $\{\xb_k\}_{k \in K}$ is a converging subsequence of unsuccessful iterates where neither \eqref{eq:restoration} nor \eqref{eq:improvement} are 
%	satisfied in tentative points, and let $\db$ be a refined direction for this subsequence. Then for $\xb^*$ limit point of $\{\xb_k\}$ we have
%	$h^\circ(\xb^*, \db) \geq 0$ and $f^\circ(\xb^*, \db) \geq 0$ whenever $h^\circ(\xb^*, \db) = 0$. In particular,
%	if $\cup_{k \in K} D_k$ is dense then $\xb^*$ is Clarke stationary for $h$. If furthermore $h$ is a suitably defined infeasibility
%	measure and $h(\xb^*) = 0$, then $\xb^*$ is a KKT stationary point for the problem.
%\end{proposition}
%\begin{proof}
%	We get $h^\circ(\xb^*; \db) \geq 0$ and $M^\circ(\xb^*, \mu; \db) \geq 0$  with the usual arguments from the assumption that \eqref{eq:restoration} and \eqref{eq:improvement} respectively are not satisfied for $k \in K$. If $h^\circ(\xb^*; \db) = 0$ then 
%	\begin{equation}
%		f^\circ(\xb^*; \db) = f^\circ(\xb^*; \db) + \mu h^\circ(\xb^*; \db) = M^\circ(\xb^*, \mu; \db) \geq 0 \, ,
%	\end{equation}
%	as desired. 
%\end{proof}

%%%%%%%%%%%%%%%%%%%%%%%%%%%%%%%%%%%%%%%%%%%%%%%%%%%%%%%%%%%%%%%%%%%%%%%%%%%%%%%%%%%%
\subsubsection{Penalty approaches}
\label{sssec:penalty}

In penalty based approaches, a direct-search scheme is applied to an auxiliary 
objective, that combines a measure of infeasibility with the objective function. 
\rev{The augmented Lagrangian technique proposed by Lewis and 
Torczon~\cite{lewis2002globally} handles} nonlinear relaxable equality constraints 
as well as bounds on the objective. This technique was later extended to handle 
linear constraints in addition to bound constraints~\cite{kolda2006generating}. 
\rev{In both cases, a directional method is used to solve the augmented Lagrangian 
subproblems.} The underlying theoretical guarantees rely on using the stepsize as a 
stopping criterion for solving the subproblems.

In another line of work, Liuzzi et al.~\cite{liuzzi2009derivative} \rev{proposes} 
an exact penalty approach to reformulate problem constrained with nonlinear, 
relaxable inequalities. The resulting penalized objective is defined on an open 
set, and tackled with a line-search method akin to those described in 
Section~\ref{ssec:linesearch}. Liuzzi et al.~\cite{liuzzi2010sequential} \rev{considers} a 
sequential penalty approach for problems with smooth objective, inequality 
constraints and bound constraints. In this approach, the inequality constraints 
other than bounds are relaxed and a line-search method is applied to minimize 
$f(\xb) + \frac{1}{\varepsilon} \rev{h}(\xb)$ under bound constraints, where 
\rev{$h$ is the smooth infeasibility measure~\eqref{eq:infeasibility2}} and 
$\varepsilon>0$ \rev{is decreased based on a condition involving an upper bound 
on the stepsizes}. The case of nonsmooth relaxable constraints combined with bounds 
was addressed by Fasano et al.~\cite{fasano2014linesearch}. Under constraint 
qualification and for small enough $\varepsilon$, it was then shown that all the 
stationary points corresponding to minimizing 
$f(\xb) + \frac{1}{\varepsilon} \rev{h}(\xb)$
\rev{with $h$ defined by~\eqref{eq:infeasibility}} are also (Clarke) 
stationary for the original problem. 

%%%%%%%%%%%%%%%%%%%%%%%%%%%%%%%%%%%%%%%%%%%%%%%%%%%%%%%%%%%%%%%%%%%%%%%%%%%%%%%
\section{\rev{Multiobjective \revn{direct-search methods}}}
\label{sec:multiobj}
%%%%%%%%%%%%%%%%%%%%%%%%%%%%%%%%%%%%%%%%%%%%%%%%%%%%%%%%%%%%%%%%%%%%%%%%%%%%%%%

\rev{In this section, we survey the main developments in extending 
\revn{direct-search schemes} 
to Multiobjective optimization (MOO) problems.
The theory of multiobjective continuous
optimization has undergone significant development since the 2000s, allowing
for further development in the context of direct-search methods. Moreover, 
multiobjective problems are ubiquitous in engineering problems, a primary 
source of applications for direct-search methods.}

\rev{Directional methods were extended to the multiobjective setting by 
Cust\'{o}dio et al.~\cite{custodio2011direct}, along with convergence to a
stationary point in the sense of Pareto. More recently, complexity bounds
were derived for a directional direct-search 
scheme~\cite{custodio2021worstcase}. Mesh-based algorithms were considered
by Audet et al.~\cite{audet2008multiobjective,audet2010mesh}, the
bi-objective version being successfully implemented in the NOMAD
software~\cite{NomadV4}. Finally, line-search techniques are amenable to 
multiobjective extension. Liuzzi et al.~\cite{liuzzi2016derivative} propose
an analysis based on new optimality conditions.}
%\subsection{\rev{Direct Multisearch}}
%\label{sec:multiobj}
\rev{We consider MOO problems of the form
\begin{equation}
\label{eq:MOO}
	\min_{\xb \in \R^n} F(\xb)\equiv(f_1(\xb),\dots, f_{\revn{l}}(\xb))^\top, \ 
\end{equation}
where $f_i:\R^n\to \R,$ with $i=1,\dots,\revn{l}$.  In this context the concept of Pareto dominance is used to compare two points. 
\begin{definition}[Pareto dominance]\label{paretodominance}
Given two points $\x,\y \in \R^n$, we say that $\x$ dominates $\y$, when  $F(\x)\leq F(\y)$, i.e., $\f_i(\x)\leq\f_i(\y)$ for all $i=1,\dots,\revn{l}$ and $\revn{F(\x)}\neq F(\y)$.
\end{definition}
A point $\revn{\hat{\xb}}$ is then a local Pareto minimizer if there exists no point  $\y\in {\cal B}(\revn{\hat{\xb}},\delta)$ s.t. $F(\y)\leq F(\revn{ \hat{\xb}})$, for some $\delta>0$. In both the nonsmooth and the smooth settings, one can define necessary optimality conditions as follows. 
\begin{definition}\label{Pareto-Clarke stationary}
  Let $F$ be Lipschitz continuous near $\revn{\hat{\xb}}$. We say that $\revn{\hat{\xb}}$ is a Pareto-Clarke stationary point if 
  \begin{equation}\label{cond:ClarkeNC}
    \forall\ \bm{d} \in \R^n, \exists\ i_d\in [1:\revn{l}]:\ f_{i_d}^\circ(\revn{\hat{\xb}};\bm{d}) \ge 0.
  \end{equation}
\end{definition}
\begin{definition}\label{Paretostationary}
  Let $F$ be continuously differentiable at $\revn{ \hat{\xb}}$. We say that $\revn{\hat{\xb}}$ is Pareto stationary, if 
  \begin{equation}\label{cond:ParetoNC}
    \forall\ \bm{d} \in \R^n, \exists\ i_d\in [1:\revn{l}]:\ \scal{\nabla f_{i_d}(\revn{ \hat{\xb}})}{\db}\ge 0.
  \end{equation}
\end{definition}
It follows that the quantity
\begin{equation}\label{Paretostatmeasure}
\mu(\x) := - \min_{\|\db\| \leq 1} \max_{i \in [1:\revn{l}]} \scal{\nabla f_{i}(\xb)}{\db}. 
\end{equation}
can be used to characterize Pareto stationarity.
In our context, a finite set of polling directions is however used at each iteration. An approximation of $\mu(\x)$ is hence often used in the analysis to take into account this aspect~\cite{custodio2021worstcase}:
\begin{equation}\label{approxParetostatmeasure}
\mu_{\bbD_k}(\x) := - \min_{\db \in \bbD_k, \|\db\| \leq 1} \max_{i \in [1:\revn{l}]} \scal{\nabla f_i(\x)}{\db}.
\end{equation}
We here focus on the Direct Multisearch (DMS) algorithm proposed in~\cite{custodio2011direct}. This approach, 
inspired by the search/poll paradigm of direct-search methods of directional type, does not aggregate any of the objective functions but rather handles a list $L_k$ of nondominated points (from which the new iterates are
chosen). The aim of DMS is to generate as many points in the Pareto front as possible from the polling
procedure itself, while keeping the whole framework general enough to accommodate other
strategies from the literature. DMS hence generalizes to multiobjective optimization (MOO) all direct-search methods of directional type. We report a simplified version of DMS in Algorithm \ref{alg:basicDMS}. It uses the strict partial order induced by the cone~$\R^m_+$. Following \cite{custodio2021worstcase}, we indicate with $D(L)\subset \R^m$ the image of the set points dominated by a list of points $L$ and $D(L;a)$ the image related to those set of points with an  $\ell_\infty$-norm  distance to $D(L)$ no larger than $a$. An iteration is successful in Algorithm \ref{alg:basicDMS} when there are modifications in the list of nondominated points, that is a new point was accepted, such that $F(\x)\notin D(L,\rho(\ssize))$.
It is crucial to emphasize that,  when dealing with MOO problems, the cone of descent directions related to all components of the objective $F$ can become arbitrarily narrow. We hence always need to consider density of the search directions at a given limit point~\cite{custodio2011direct,custodio2021worstcase} to guarantee convergence.
}

\newenvironment{algocolor}{%
   \setlength{\parindent}{0pt}
%   \itshape
   \color{black}
}{}

\begin{algorithm}[h]
   \begin{algocolor}
	\caption{Basic Direct Multisearch method}
	\label{alg:basicDMS}
	\begin{algorithmic}[1]
        \par\vspace*{0.1cm}
        
		\STATE \textbf{Inputs:} Starting point/incumbent solution 
		$\x_0\in\mathbb{R}^{n}$, initial stepsize $\ssize_0> 0$, $L_0=\{(\x_0,\ssize_0)\}$,\\ 
        stepsize update parameters $0 < \theta < 1 \le \gamma$, forcing 
        function $\rho: \Rplus \rightarrow \Rplus$.
		\FOR{$k=0,1,2,\ldots$}
			\STATE Select a set $\bbD_k$ of poll directions and evaluate $F(\x)$  at the points $\{\x_k+\ssize_k \bm{d}:\ \bm{d}\in\bbD_k\}$
		      \STATE Compute $L_{k+1}$ by removing dominated points from $L_k\cup \{(\x_k+\ssize_k \bm{d}, \alpha_k):\ \bm{d}\in\bbD_k\}$  
              \IF{$L_{k+1}\neq L_k$} 
                \STATE Declare the iteration as successful, replace new pairs with $(\x_k+\alpha_k \bm{d}, \gamma \ssize_k)$ in $L_{k+1}$.
            \ELSE
			    \STATE Declare the iteration as unsuccessful, replace   $(\x_k, \alpha_k)$ with $(\x_k, \theta \ssize_k)$ in $L_{k+1}$.
            \ENDIF
		\ENDFOR
    
	\end{algorithmic}
\end{algocolor}
\end{algorithm}
\rev{
Convergence results can be established under the following assumptions, that generalize those made in single-objective optimization.
\begin{assumption}\label{ass:LCGmulti}
For all \( i \in [1:\revn{l}] \), the function \( f_i \) is continuously differentiable with Lipschitz continuous gradient with constant \( \revn{M_i} \). Set \( \revn{M_{\text{max}}} = \max_{i \in [1:\revn{l}]} \revn{M_i} \).
\end{assumption}
\begin{assumption}\label{ass:boundmulti}
For all \( i \in [1:\revn{l}] \), the function \( f_i \) is lower and upper bounded in \( \{ \x \in \mathbb{R}^n : F(\x) \notin D(\{\x_0\}) \} \), with lower bound \( f_i^{\min} \) and upper bound \( f_i^{\max} \). Let 
\[
F^{\min} := \min\{ f_1^{\min}, \dots, f_{\revn{l}}^{\min} \} \quad \text{and} \quad F^{\max} := \max\{ f_1^{\max}, \dots, f_{\revn{l}}^{\max} \}.
\]
\end{assumption}
\begin{assumption}\label{ass:compmulti}
The set \( \{ \x \in \mathbb{R}^n : F(\x) \notin D(\{\x_0\}) \} \) is compact.
\end{assumption}
We further assume, without any loss of generality that all the positive spanning sets considered have normalized directions. Similarly to the analysis carried out in the single-objective case, we now report a result that connects at each iteration $k$ the approximate stationarity measure $\mu_{\bbD_k}(\x_k)$  and the stepsize $\ssize_k$, whose proof is given in \cite{custodio2021worstcase}:
\begin{proposition}\label{prop:DMSsmooth}
Let Assumption \ref{ass:LCGmulti} hold. Consider that 
    Algorithm~\ref{alg:basicDMS} is applied to the minimization of~$F$, with $\bbD_k$ positive spanning set. 
    Suppose that the $k$th iteration of Algorithm~\ref{alg:basicDMS} is 
    unsuccessful. Then, 
	\begin{equation} \label{eq:deltakDMS}
		\mu_{\bbD_k}(\x_k)
		\le
		\ssize_k \frac{\revn{M_{\text{max}}}}{2} + \frac{\rho(\ssize_k)}{\ssize_k} 
		\,.
	\end{equation}
\end{proposition}
The following assumption, which connects the exact and approximate stationarity measure, is also needed in the analysis.
\begin{assumption}\label{connectstmeasuremulti}
There exists \( C_1 > 0 \) such that 
\[
|\mu_{\bbD_k}(\x_k) - \mu(\x_k)| \leq C_1 \mu_{\bbD_k}(\x_k), \quad \forall\ k \geq 0.
\]
\end{assumption}
A direct application of Proposition~\ref{prop:DMSsmooth} under this assumption yields the following convergence result.
\begin{theorem}  \label{th:liminfdirdms}
    Suppose that Algorithm~\ref{alg:basicDMS} is applied to  $F$ under Assumptions~~\ref{ass:LCGmulti}, \ref{ass:boundmulti} and \ref{connectstmeasuremulti}. \rev{Suppose also 
    that the sequence $\{\bbD_k\}$ is such that $\bbD_k$ is a positive spanning set for all $k\geq 0$. Let $\mathcal{K}$ be a index set of unsuccessful iterations such that 
    $\lim_{k\in \mathcal{K}} \alpha_k \rightarrow 0$. Then,}
	\begin{equation}
    \label{eq:liminfdirdms}
		\lim_{k \in \mathcal{K}} \mu(\x_k)= 0.
	\end{equation}
\end{theorem}
As noticed in \cite{custodio2021worstcase},  Assumption \ref{ass:compmulti} and the use of sufficient decrease to accept new nondominated points guarantee the existence of a refining subsequence $\{\x_k\}_\mathcal{K}$  such that $\{\ssize_k\}_\mathcal{K}$ converges to zero (see Definition~\ref{def:refining}). We would also like to highlight that Assumption \ref{connectstmeasuremulti} basically requires nonnegativity of the approximate measure $\mu_{\bbD_k}(\x_k)$ at each $k$. This can only be obtained if a rich enough set $\bbD_k$ of poll directions is available.}

\rev{We now state the main convergence result for the basic DMS algorithmic scheme we reported above. Similarly to what we have seen in the single-objective case, we consider refining subsequences and assume density in the unit sphere for the related refining directions. The proof is directly derived from~{\rdj \cite[Theorem~4.9]{custodio2011direct}}.
\begin{theorem} \label{t:cDMS}
    Suppose that Algorithm~\ref{alg:basicDMS} is applied to a Lipschitz 
    continuous function $F$. Let $\mathcal{K}$ be a sequence of unsuccessful steps such 
    that $\{\xb_k\}_{k \in \mathcal{K}}$ is a refining subsequence converging to $\revn{ \hat{\xb}}$. 
    If the corresponding set of refining directions is dense in the unit sphere, 
    then $\revn{\hat{\xb}}$ is a Pareto-Clarke stationary point. If in addition $F$ is continuously differentiable, then $\revn{ \hat{\xb}}$ is a Pareto stationary point.
\end{theorem}
It is important to note that the density of the refining directions is also needed in the smooth case. As mentioned before, this is obviously due to the fact that the cone of descent directions for the components of the function $F$ can become arbitrarily narrow.
}

\rev{
In addition to guaranteeing global convergence, the 
sufficient decrease condition also provides DMS  with complexity 
bounds that quantify how many iterations or function evaluations are needed in the worst case scenario to reach an approximate optimality 
criterion, such as $ \mu(\x_k)\le \epsilon$. This is once again comparable to what we observed in the single-objective case, though we note that the analysis is somewhat more involved. Indeed, it is based on
linked sequences, i.e., sequences of pairs $(\x_{j_k},\ssize_{j_k})$, such that for any $k=1,2,\dots$ the pair $(\x_{j_k},\ssize_{j_k})\in L_k$ is generated at iteration $k-1$ of Algorithm \ref{alg:basicDMS} by the pair
$(\x_{j_{k-1}},\ssize_{j_{k-1}})\in L_{k-1}$.
\begin{theorem}[\cite{custodio2021worstcase}]
%{\rdj (Help with reference please? It looks like~\cite{dodangeh2016optimal} cite below provides a particular case...)} 
\label{t:DMScomplexity}
    Let the assumptions of Theorem~\ref{th:liminfdirdms} hold. 
    Suppose further that $\rho(\alpha) = c\,\alpha^2$ for some $c > 0$, 
    and that $|\bbD_k| \leq \revn{m}$ for 
    every $k \in \mathbb{N}$. Then, for any 
    $\epsilon>0$, \rev{Algorithm~\ref{alg:basicDMS} reaches an iterate 
    $\xb_k$ such that $\mu(\x_k) \le \epsilon$ in at most 
    \begin{equation}
    \label{eq:n2complexityDMS}
    	\cO\left(|L(\epsilon)|\epsilon^{-2m}\right)
    	\ \mbox{iterations} 
    	\quad \mbox{and} \quad
    	\cO\left(\revn{m}|L(\epsilon)|\epsilon^{-2m}\right)
    	\ \mbox{function evaluations},
    \end{equation}
    where $L(\epsilon)$ is the set of linked sequences between the first unsuccessful iteration and the iteration before the one that satisfies the criticality condition.
    }
\end{theorem}
The reported bounds do not match the $\cO\left(\epsilon^{-2}\right)$ complexity bound that the gradient method for MOO problems has (see \cite{fliege2019complexity} for further details). The main reason behind this difference is that an iteration is declared successful in DMS when a sufficient decrease is guaranteed for at least one component of $F$, while in the gradient method the iterate is updated when a suitable decrease is obtained with respect to all components. The dependence on the term $|L(\epsilon)|$ is due instead to the fact that DMS tries to approximate the whole Pareto front using the list $L_k$, while the gradient method in \cite{fliege2019complexity} only finds one Pareto stationary point. 
}
\revn{ As a final remark, we would like to highlight that for the case of (strongly) convex optimization problems, i.e., where all components of $F$ are (strongly) convex, a complexity analysis should ideally expand upon the findings presented in \cite{dodangeh2016worst}, which focus on the single-objective scenario. However, it still remains as an open question whether similar complexity bounds to those derived in \cite{dodangeh2016worst} could be established for the multiobjective case.}

\rev{The handling of constraints in a MOO problem was already addressed in \cite{audet2008multiobjective,audet2010mesh,custodio2011direct}, where an extreme barrier approach was used for this purpose. An exact penalty approach was analyzed in \cite{liuzzi2016derivative} for handling general constraints in a linesearch-based  algorithmic framework. To the best of the authors' knowledge, no existing work provides a detailed account of a Direct-Search algorithm for stochastic MOO problems.}
%%%%%%%%%%%%%%%%%%%%%%%%%%%%%%%%%%%%%%%%%%%%%%%%%%%%%%%%%%%%%%%%%%%%%%%%%%%%%%%
\section{Conclusion}
\label{sec:conc}
%%%%%%%%%%%%%%%%%%%%%%%%%%%%%%%%%%%%%%%%%%%%%%%%%%%%%%%%%%%%%%%%%%%%%%%%%%%%%%%

Since their formal introduction more than 60 years ago, direct-search 
algorithms have been endowed with a rich convergence theory. \rev{In addition 
to possessing convergence and complexity guarantees in the smooth,  
unconstrained setting,} direct-search methods can also be designed to cope 
with nonsmoothness, \rev{noisy function values, constraints of various types 
as well as multiple objectives.} In this survey, we provided key results 
\rev{as well as pointers to the literature} that we hope 
\revn{have convinced} the readers of the strong theoretical foundations upon 
which modern direct-search algorithms are built.

Our presentation has highlighted several \rev{lines of research that will in 
our opinion shape} the future of direct-search techniques. Recent advances in 
analyzing randomized direct-search techniques pave the way for scaling up 
those algorithms, possibly through combination with other tools from 
high-dimensional statistics. \rev{Modern constrained optimization techniques 
such as manifold theory provide principled ways of handling common 
constraints}\revnn{, and may lead to novel analyzes departing from the use of 
standard nonlinear programming tools.} \rev{Tackling noise in function values is likely to keep growing 
in importance to address problems involving massive amounts of data samples 
such as hyperparameter tuning}\revnn{, while much remains to be done in 
obtaining a comprehensive, probabilistic analysis of direct-search algorithms in that setting.} \revnn{Overall, we believe that these three} research directions are 
likely to form the basis of next-generation direct-search methods.

%\paragraph{Acknowledgments} We deeply thank the Editor-in-Chief, the Associate Editor and two anonymous reviewers for their patience, time and effort while reviewing this paper. Their comments lead to numerous improvements and corrections.

%%%%%%%%%%%%%%%%%%%%%%%%%%%%%%%%%%%%%%%%%%%%%%%%%%%%%%%%%%%%%%%%%%%%%%%%%%%%%%%%%%%%
\bibliography{references2}
\bibliographystyle{abbrvnat}
%%%%%%%%%%%%%%%%%%%%%%%%%%%%%%%%%%%%%%%%%%%%%%%%%%%%%%%%%%%%%%%%%%%%%%%%%%%%%%%%%%%%

% \clearpage
% % This can be deleted in the version that appears, it is only needed for the Argonne PANDA submission and ArXiv submission:
% \vspace{2cm}
% \framebox{\parbox{\linewidth}{
% The submitted manuscript has been co-created by UChicago Argonne, LLC, Operator of 
% Argonne National Laboratory (``Argonne''). Argonne, a U.S.\ Department of 
% Energy Office of Science laboratory, is operated under Contract No.\ 
% DE-AC02-06CH11357. 
% The U.S.\ Government retains for itself, and others acting on its behalf, a 
% paid-up nonexclusive, irrevocable worldwide license in said article to 
% reproduce, prepare derivative works, distribute copies to the public, and 
% perform publicly and display publicly, by or on behalf of the Government.  The 
% Department of Energy will provide public access to these results of federally 
% sponsored research in accordance with the DOE Public Access Plan. 
% http://energy.gov/downloads/doe-public-access-plan.}}

\end{document}